\documentclass[11pt,a4paper,reqno,svgnames]{amsart}
\usepackage{amsthm,amssymb,amsfonts,mathtools}
\usepackage{mathtools}
\mathchardef\mhyphen="2D
\usepackage{extarrows}
\usepackage[color]{changebar}
\cbcolor{red}
\usepackage{scalefnt}
\usepackage{latexsym}
\usepackage{tikz}
\usepackage{scalefnt}

\usepackage{youngtab}
\usepackage{latexsym}
\usepackage{amssymb}
\usepackage{mathrsfs}
\usepackage{amsthm}
\usepackage{enumitem}
\usepackage{fullpage}
\usepackage{stmaryrd}
 \usepackage{color}
\usepackage{tikz}
\usepackage{xcolor}
\usepackage[plainpages=false, pdfpagelabels,  colorlinks=true, pdfstartview=FitV, linkcolor=DarkBlue, citecolor=DarkRed, urlcolor=blue]{hyperref}

\usepackage{cleveref}

\crefname{definition}{Definition}{Definitions}
\crefname{theorem}{Theorem}{Theorems}
\crefname{proposition}{Proposition}{Propositions}
\crefname{lemma}{Lemma}{Lemmas}
\crefname{corollary}{Corollary}{Corollaries}
\crefname{conj}{Conjecture}{Conjectures}
\crefname{section}{Section}{Sections}
\crefname{subsection}{Subsection}{Subsections}
\crefname{eg}{Example}{Examples}
\crefname{figure}{Figure}{Figures}
\crefname{rem}{Remark}{Remarks}
\crefname{rmk}{Remark}{Remarks}
\crefname{equation}{equation}{equation}

\Crefname{definition}{Definition}{Definitions}
\Crefname{theorem}{Theorem}{Theorems}
\Crefname{proposition}{Proposition}{Propositions}
\Crefname{lemma}{Lemma}{Lemmas}
\Crefname{corollary}{Corollary}{Corollaries}
 \Crefname{section}{Section}{Sections}
\Crefname{subsection}{Subsection}{Subsections}
\Crefname{eg}{Example}{Examples}
\Crefname{figure}{Figure}{Figures}
\Crefname{rem}{Remark}{Remarks}
\Crefname{rmk}{Remark}{Remarks}

\theoremstyle{plain}
\newtheorem{theorem}{Theorem}[section]
\newtheorem{proposition}[theorem]{Proposition}
\newtheorem{lemma}[theorem]{Lemma}
\newtheorem{corollary}[theorem]{Corollary}

\newtheorem*{Acknowledgements*}{Acknowledgements}

\theoremstyle{definition}
\newtheorem{definition}[theorem]{Definition}
\newtheorem{example}[theorem]{Example}
\newtheorem{remark}[theorem]{Remark}

 \DeclareMathOperator{\Std}{Std}

\DeclareMathOperator{\rad}{rad}

\DeclareMathOperator{\res}{res}

\newcommand{\stw}{\mathsf{w}} 
\newcommand{\sts}{\mathsf{s}}   
 \newcommand{\stu}{\mathsf{u}}
\newcommand{\stv}{\mathsf{v}}
\newcommand{\stt}{\mathsf{t}}  
\newcommand{\ZZ}{\mathbb{Z}}   
\newcommand{\QQ}{\mathbb{Q}}

\newcommand{\FF}{\mathbb{F}}
\newcommand{\suchthat}{\;\ifnum\currentgrouptype=16 \middle\fi|\;} 
\newcommand{\obpointla}{(\lambda,k)}

\newcommand{\pointmu}{(\mu,k)}

\newcommand{\pointla}{(\lambda,k)}

\renewcommand{\ge}{\geqslant}
\renewcommand{\ge}{\geqslant}
\renewcommand{\geq}{\geqslant}
\renewcommand{\le}{\leqslant}
\renewcommand{\leq}{\leqslant}
\renewcommand{\unrhd}{\trianglerighteqslant}
\renewcommand{\unlhd}{\trianglelefteqslant}
\renewcommand{\succeq}{\succcurlyeq}

\newcommand\algebra[2]{{P}_{#1}^{#2}}
\newcommand\subalgebra{P^R_{\vartriangleright  (\lambda,k)}}
\newcommand\standard[2]{{\Delta}_{#1}^{#2}}
\newcommand\simple[2]{{L}_{#1}^{#2}}
\def\hods{\unskip\kern.55em\ignorespaces}

\begin{document}
\usetikzlibrary{matrix,arrows,decorations.pathreplacing,backgrounds,decorations.markings}

\title{Simple modules for the partition algebra   and \\ monotone convergence of  Kronecker coefficients  }
\author{C. Bowman}
\author{M. De Visscher}
\author{J. Enyang}
\address{Department of Mathematics, City University London, Northampton Square, London EC1V 0HB, United Kingdom }
 
\maketitle

\begin{abstract}
We construct bases of the simple modules for partition algebras which are indexed by paths in an alcove geometry.  
This allows us to give a concrete interpretation (and new proof) of  the monotone convergence property for Kronecker coefficients 
using  stratifications of the cell modules of the partition algebra.  
\end{abstract}

\section*{Introduction}

A fundamental problem in the representation theory of the symmetric group is to describe the coefficients in the decomposition of the tensor product of two Specht modules.     In \cite{MR3314819}, the first two authors and Orellana proposed a new approach to this problem by using the Schur--Weyl duality between the symmetric group $\mathfrak{S}_n$ and the partition algebra  $P_{2k}^\mathbb{Q}(n)$.  The Kronecker coefficients are interpreted as the decomposition multiplicities which arise in restricting a simple $P_{2k}^\mathbb{Q}(n)$-module   to a Young subalgebra.  

The purpose of this article is to construct bases of the simple modules of partition algebras 
   and begin exploring the applications of these bases to the study of the Kronecker coefficients.  
We proceed by  embedding the branching graph  of the partition algebra into a parabolic alcove geometry of type $A_{r-1} \subseteq A_{r}$.  
We then construct bases of the simple   modules which are indexed by paths in this graph 
 satisfying a certain geometric condition.  
In the case that the algebra is semisimple, the resulting bases are equal to those constructed in \cite{MR3092697}.  
 That this geometry governs the representation theory of the partition algebras is perhaps surprising,
  but could be explainable through interactions with the associated parabolic category $\mathcal{O}$, via Deligne's tensor category (see \cite{Aiz} and \cite{CO11}).  


We turn now to the Kronecker coefficients. One   attempt to understand  these coefficients is via  their limiting behaviour and stability properties.  
Murnaghan observed that as  we increase the length of the first row of the indexing partitions,
 the sequence of Kronecker coefficients obtained stabilises. 
  The limits of these sequences are known as the \emph{stable  Kronecker coefficients}.
This stability has been proven 
using invariant theory \cite{MR0095209},  
 geometric methods \cite[Section 3.4]{MR1243152},
  and 
    by means of vertex operators \cite[Section 3]{MR1128013}.  
%
%
%
%
%
%
In \cite{MR3314819} we observed that   increasing the length of the first row of the indexing partitions corresponds to increasing the parameter $n$ for the partition algebra. The stability is then naturally explained by the fact that the partition algebra is semisimple for large values of $n$. Moreover, the stable Kronecker coefficients are  given a concrete 
representation theoretic interpretation as the decomposition multiplicities of the restriction of a cell module  for the partition algebra to a Young subalgebra.  

The bases  constructed in this paper allow  us, for  the first time, to study simple modules  for partition algebras directly.  
 In the past,   one had to   study the simple modules indirectly via the  cell modules of the partition algebra.  
Correspondingly, almost all known formulas \cite{111,ROSAANDCO,Rosas,RW} for computing Kronecker products proceed indirectly 
 by  
(implicitly) expressing the Kronecker coefficient as a signed sum of  stable Kronecker coefficients. 
 Our main result gives explicit bases for the simple modules; the action of the generators of the partition algebra on these bases have been described in \cite{MR3092697}. This gives a direct approach to studying the  non-stable  Kronecker coefficients  (which we will explore in further work) and for studying Murnaghan's stability   phenomena (see below).  

 In \cite{MR1243152}  Brion  showed  that,  as  we increase the length of the first row of the indexing partitions,
 the sequence of Kronecker coefficients obtained is {\em weakly increasing} (in addition to having a stable limit).  
 This monotone convergence property has been proved using geometric methods  \cite{MR1243152} and 
 by analysing   integer points in polyhedra \cite{Stembridge}.  In this paper, 
 we provide a new proof of this monotone convergence property 
   by providing a  manifestly positive interpretation of the    differences between adjacent terms  in this sequence.
 Namely, we provide a stratification of 
   a given cell module; the layers of this stratification  
  decompose according to these differences.  
 This   uses  only simple methods in combinatorial representation theory and  answers a question posed to us by Briant.  

 \bigskip

\textbf{Notation:} 
The focus of this paper is the representation theory of the partition algebra over the field $\mathbb{Q}$ and with integer parameter $n$. 
However, we wish to use results from Enyang \cite{MR3092697} on the seminormal representations of the partition algebra with parameter $z$ over the field $\mathbb{Q}(z)$. 
We will relate these in the usual manner.
Let $z$ be a variable and define $R:=\mathbb{Z}[z]$ with field of fraction $\mathbb{F}:=\mathbb{Q}(z)$.
Given $n\in \mathbb{Z}_{\geq 0}$, we define $\mathcal{O}_n$ to be the localisation of $R$ at the prime ideal $\mathfrak{p}=(z-n)$. Then we have natural embeddings $R\hookrightarrow  \mathcal{O}_n \hookrightarrow \mathbb{F}$ and  projection map $\pi_n : \mathcal{O}_n \rightarrow \mathbb{Q}$ given by specialising to $z=n$, giving both $\mathbb{F}$ and $\mathbb{Q}$ the structure of $\mathcal{O}_n$-bimodules.
 Throughout the paper, all modules for the partition algebras (over $R$, $\mathbb{F}$, $\mathcal{O}_n$ and $\mathbb{Q}$) will be right modules.

\bigskip

  The first three sections are not new but recall all the necessary background.
In Section 1, we define the partition algebra over $R$ and recall its cellular structure. In particular, we recall the properties of the Murphy-type basis constructed inductively on the branching graph by Goodman and Enyang in \cite{EG:2012}. In Section 2, we consider the representations of the partition algebra over $\mathbb{F}$. In particular, we recall some of the properties of the seminormal bases for cell modules constructed by Enyang \cite{MR3092697}. In Section 3--7 we work over the field $\mathbb{Q}$.
In Section 3, we apply the general framework developed by Mathas to generalise the seminormal bases to the partition algebra over $\mathbb{Q}$ with integer parameter.
In Section 4 we introduce a reflection geometry (for a fixed parameter $n$) on the branching graph for the partition algebra and reinterpret results due to Martin on the representation theory of the partition algebra over $\mathbb{Q}$ in this geometrical setting. We also develop properties of this geometry which will be needed in the following sections.
In Section 5 we study the restriction of cell and simple modules. This will be used in Section 6 to describe a basis for the radical of cell modules, and hence also for simple modules. The main results are Theorem 6.5 and Corollary 6.6. These are then used in Section 7 to give a new interpretation (and a new proof) of the monotone convergence of Kronecker coefficients.

\begin{Acknowledgements*}
The authors 
  are grateful for the financial support received from 
the {Royal Commission for the Exhibition of 1851} and  {EPSRC}  {grant  EP/L01078X/1}. 
The authors would also like to thank Emmanuel Briant for some useful discussions. 
\end{Acknowledgements*}

\section{The Partition algebra: Branching graph and Cellularity}

Let $z$ be a variable and set  $R=\mathbb{Z}[z]$. 
For a fixed $k\in\mathbb{Z}_{\geq 0}$, we define the partition algebra $\algebra{2k}{R}(z)$ to be the set of $R$-linear combinations of set-partitions of $\{1,2,\dots,k,\bar1,\bar2,\dots,\bar k\}$. (For $k=0$ we set $P_0^R(z)=R$.) We call each connected component of a set-partition a {\sf block}.
 For example,
\[
d=\{\{\overline1, \overline2, \overline4,  {2},  {5}\}, \{\overline3\}, \{\overline5, \overline6, \overline7,  {3},  {4},
 {6},  {7}\}, \{\overline8,  {8}\}, \{ {1}\}\}
\]
is a set-partition (for $k=8$) with 5  blocks.

A  set-partition can be represented 
 by a  diagram consisting of a frame with $k$ distinguished points on the northern and southern boundaries, which we call vertices.  We number the southern vertices from left to right by $1,2,\ldots, k$ and the northern vertices similarly by $\bar{1},\bar{2},\ldots, \bar{k}$ and connect two vertices by a path if they belong to the same block.  Note that such a diagram is not uniquely defined, two diagrams representing the set-partition $d$ above are given in \hyperref[2diag]{Figure~\ref*{2diag}}.

\begin{figure}[ht]
\begin{tikzpicture}[scale=0.5]
  \draw (0,0) rectangle (8,3);
  \foreach \x in {0.5,1.5,...,7.5}
    {\fill (\x,3) circle (2pt);
     \fill (\x,0) circle (2pt);}
  \begin{scope} 
    \draw (0.5,3) -- (1.5,0);
    \draw (7.5,3) -- (7.5,0);
    \draw (4.5,3) -- (2.5,0);
    \draw (0.5,3) arc (180:360:0.5 and 0.25);
    \draw (1.5,3) arc (180:360:1 and 0.25);
     \draw (4.5,0) arc (0:180:1.5 and 1);
    \draw (5.5,0) arc (0:180:1 and .7);
    \draw (3.5,0) arc (0:180:.5 and .25);
    \draw (6.5,0) arc (0:180:0.5 and 0.5);
    \draw (4.5,3) arc (180:360:0.5 and 0.25);
    \draw (5.5,3) arc (180:360:0.5 and 0.25);
      \draw (2.5,-0.5) node {$3$};   \draw (2.5,3.5) node {$\bar{3}$};
                  \draw (1.5,-0.5) node {$2$};   \draw (1.5,3.5) node {$\bar{2}$};
                           \draw (0.5,-0.5) node {$1$};   \draw (0.5,3.5) node {$\bar{1}$};
         \draw (3.5,-0.5) node {$4$};   \draw (3.5,3.5) node {$\bar{4}$};
                  \draw (4.5,-0.5) node {$5$};   \draw (4.5,3.5) node {$\bar{5}$};
                           \draw (5.5,-0.5) node {$6$};   \draw (5.5,3.5) node {$\bar{6}$};
         \draw (6.5,-0.5) node {$7$};   \draw (6.5,3.5) node {$\bar{7}$};         \draw (7.5,-0.5) node {$8$};   \draw (7.5,3.5) node {$\bar{8}$};
   \end{scope}
\end{tikzpicture}
\quad \quad 
\begin{tikzpicture}[scale=0.5]
  \draw (0,0) rectangle (8,3);
  \foreach \x in {0.5,1.5,...,7.5}
    {\fill (\x,3) circle (2pt);
     \fill (\x,0) circle (2pt);}
  \begin{scope}     \draw (0.5,3) -- (1.5,0);
    \draw (7.5,3) -- (7.5,0);
    \draw (5.5,3) -- (6.5,0);  \draw (1.5,0) -- (3.5,3);\draw (3.5,3) -- (4.5,0);
    \draw (0.5,3) arc (180:360:0.5 and 0.25);
     \draw (5.5,0) arc (0:180:1 and .7);
    \draw (3.5,0) arc (0:180:.5 and .25);
    \draw (6.5,0) arc (0:180:0.5 and 0.5);
    \draw (4.5,3) arc (180:360:1 and 0.7);
    \draw (5.5,3) arc (180:360:0.5 and 0.25);
          \draw (2.5,-0.5) node {$3$};   \draw (2.5,3.5) node {$\bar{3}$};
                  \draw (1.5,-0.5) node {$2$};   \draw (1.5,3.5) node {$\bar{2}$};
                           \draw (0.5,-0.5) node {$1$};   \draw (0.5,3.5) node {$\bar{1}$};
         \draw (3.5,-0.5) node {$4$};   \draw (3.5,3.5) node {$\bar{4}$};
                  \draw (4.5,-0.5) node {$5$};   \draw (4.5,3.5) node {$\bar{5}$};
                           \draw (5.5,-0.5) node {$6$};   \draw (5.5,3.5) node {$\bar{6}$};
         \draw (6.5,-0.5) node {$7$};   \draw (6.5,3.5) node {$\bar{7}$};         \draw (7.5,-0.5) node {$8$};   \draw (7.5,3.5) node {$\bar{8}$};
  \end{scope}
\end{tikzpicture}
  \caption{Two representatives of the set-partition $d$.}
\label{2diag}
\end{figure}
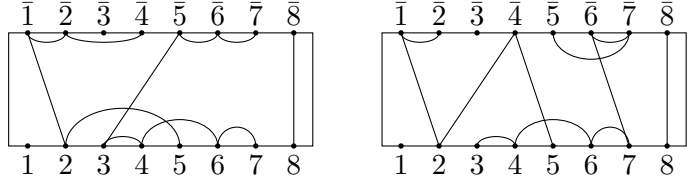

We define the product $x  y$ of two diagrams $x$ and $y$ using
the concatenation of $x$ above $y$, where we identify the southern
vertices of $x$ with the northern vertices of $y$.  If there are $t$
connected components consisting only of middle vertices, then the
product is set equal to $z^t$ times the diagram with the middle
components removed. Extending this by linearity defines a
multiplication on $\algebra{2k}{R}(z)$.

We let $\algebra{2k-1}{R}(z)$ denote the subspace of $\algebra{2k}{R}(z)$
with basis given by all set-partitions such that $k$ and $\overline{k}$
belong to the same block.  The subspace $\algebra{2k-1}{R}(z)$ is closed
under the multiplication and therefore is a subalgebra of
$\algebra{2k}{R}(z)$. We also view $P_{2k}^R(z)$ as a subalgebra of $P_{2k+1}^R(z)$ by adding to each diagram an additional block consisting of $\{2k+1, \overline{2k+1}\}$. So we obtain a tower of algebras
$$P_0^R(z) \subset P_1^R(z) \subset \ldots \subset P_{k-1}^R(z) \subset P_k^R(z) \subset \ldots$$
and we can define restriction functors ${\rm res}^k_{k-1}$ from the category of $P_k^R(z)$-modules to the category of $P_{k-1}^R(z)$-modules.

There is an anti-isomorphism  
$\ast$ on $\algebra{k}{R}(z)$ given by flipping a partition diagram through its horizontal axis.

\bigskip

The representation theory of the partition algebra can be described in terms of a directed graph, called the branching graph,  with vertices given by partitions (to be distinguished from the set-partitions considered earlier).
 
 Let $l$ denote a non-negative integer. A {\sf partition} of $l$, denoted $\lambda \vdash l$, is a weakly decreasing sequence $\lambda=(\lambda_1,\lambda_2,\dots)$ of non-negative integers such that $\sum_{i\ge 1}\lambda_i=l$. We denote by $\varnothing $ the unique partition of $0$.  If $\lambda$ is a partition, we will also write $|\lambda|=\sum_{i\ge 1}\lambda_i$. 
 With a partition, $\lambda$, is associated its {\sf Young diagram}, which is the set of nodes
\[[\lambda]=\left\{(i,j)\in\mathbb{Z}_{>0}^2\ \left|\ j\leq \lambda_i\right.\right\}.\]
  Given a node $a\in [\lambda]$ specified by $i,j\geq1$, we say the node has {\sf content}  $j-i$ and write $c(a)=j-i$.  
The diagram $[\lambda]$ is frequently represented as an array of boxes with $\lambda_i$ boxes on the $i$-th row. For example, if $\lambda=(3,2)$, then $$[\lambda]=\text{\tiny\Yvcentermath1$\yng(3,2)$}\;.$$   We will identify the partition $\lambda$ with its Young diagram and write $\lambda$ in place of $[\lambda]$. We write $\lambda \subseteq \mu$ when $[\lambda]\subseteq [\mu]$.
 Let $\lambda$ be  a partition. A node $(i,j)$ is an {\sf addable} node of $\lambda$ if $(i,j)\not\in\lambda$ and $\mu=\lambda \cup{\lbrace}(i,j){\rbrace}$ is a partition. We also refer to $(i,j)$  as a {\sf removable} node of $\mu$. We let $A(\lambda)$ and $R(\lambda)$ respectively denote the set of addable nodes and removable nodes of $\lambda$.

The {\sf branching graph}, $\mathcal{Y}$, for the partition algebra is defined as follows.  We take the  
vertex set $\cup_{k\geq 0} \mathcal{Y}_k$ where
$$\mathcal{Y}_k = \{(\lambda, k) \mid  \lambda \vdash l\leq \lfloor k/2 \rfloor \}. $$
We call $\mathcal{Y}_k$ the set of vertices on level $k$.

For $(\mu,k-1)\in \mathcal{Y}_{k-1}$ and $(\lambda,k)\in \mathcal{Y}_k$ we have an edge $(\mu,k-1)\rightarrow (\lambda,k)$ if and only if either $\lambda = \mu$, or $k$ is even and $\lambda = \mu \cup \{a\}$ for some $a\in A(\mu)$, or $k$ is odd and $\lambda  = \mu \setminus \{a\}$ for some $a\in R(\mu)$.

   The first few levels
of $\mathcal{Y}$    are given in \hyperref[brancher]{Figure~\ref*{brancher}} where, to simplify the picture, we have displayed the levels on the left hand side and  the corresponding partitions on each level.
\begin{figure}[ht!]
$$  \scalefont{0.8}
\begin{tikzpicture}[scale=0.5]
          \begin{scope}    \draw (0,3) node { $  \varnothing $  };   
  \draw (-3,0) node {   $  1$  };     \draw (-3,3) node {   $ 0$  };   
    \draw (-3,-3) node {   $ 2$  };     \draw (-3,-6) node {   $ 3$  };   
    \draw (-3,-9) node {   $ 4$  };     
              \draw (0,0) node {   $ \varnothing $  };   
    \draw (0,-3) node   {   \text{	$ \varnothing  $	}}		;    \draw (+3,-3) node   {  
     $ (1) $	
    }		;
      \draw (0,-6) node   {   \text{	$ \varnothing$	}};
     \draw (3,-6) node   {  $ (1) $	 };
    \draw [<-] (0.0,1) -- (0,2);     \draw [->] (0.0,-1) -- (0,-2);        
                \draw [->] (1,-1) -- (2,-2);    \draw [->] (2,-4) -- (1,-5);       \draw [->] (0,-4) -- (0,-5);   \draw [->] (3,-4) -- (3,-5);
  \draw [->] (0,-7) -- (0,-8);   \draw [->] (03,-7) -- (3,-8); 
    \draw [->] (1,-7) -- (2,-8);   \draw [->]  (5,-6.5) -- (8,-8);   \draw [->] (4,-7) -- (5,-8); 
     \draw (+0,-9) node   {   	  $ \varnothing $  };                 
          \draw (+3,-9) node   {  $ {(1)} $	 }		;
             \draw (+6,-9) node              {    $  {(2)}  $	 	 }		;
                          \draw (+9,-9) node            {    $ {(1^2)} $	 	 }			;
    \end{scope}\end{tikzpicture}
    $$
\caption{The branching graph of the partition algebra up to level 4.}
\label{brancher}
\end{figure}
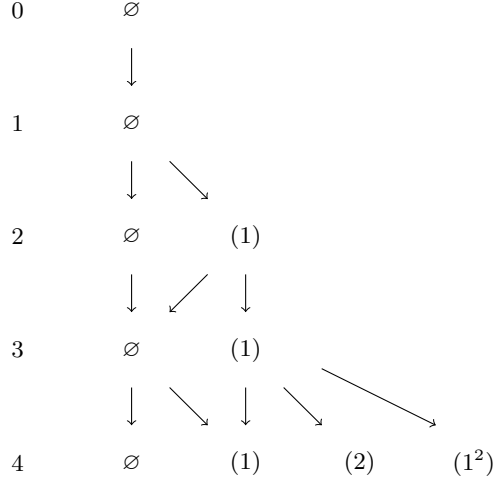

\begin{definition}\label{ordering}
For $k\in {\mathbb{Z}_{\geq 0}}$, we  denote by $\Std_k$ the set of all paths on the branching graph $\mathcal{Y}$ starting at $(\varnothing , 0)$ and ending at a vertex on level $k$. For $(\lambda,k)\in \mathcal{Y}_k$ we denote by $\Std_k(\lambda)$ the subset of $\Std_k$ consisting of all paths ending at $(\lambda,k)$. If $\stt\in \Std_k(\lambda)$ we write it as a sequence of vertices
 $$
\stt = (\stt(0),\stt(1)  ,   \stt(2), \dots, \stt(k-1),\stt(k) )
 $$
where $\stt (0) = (\varnothing ,0)$ and $\stt (k) = (\lambda,k)$. To simplify notation we will sometimes identify $\stt(k) = (\lambda,k)$ with $\lambda$.

For $0 \leq r \leq k$ and $\stt \in \Std_k$  we define  $\stt {\downarrow}_r  \in \Std_r$, the truncation of $\stt $ to level $r$,  given by
$$\stt {\downarrow} _r= (\stt (0), \stt (1), \ldots , \stt (r)).$$
\end{definition}

\begin{definition}
If $\lambda$ and $\mu$ are partitions, we write $\lambda \trianglerighteq \mu$ if either $|\lambda|<|\mu|$, or $|\lambda|=|\mu|$ and 
$$\sum_{i=1}^r \lambda_i \geq \sum_{i=1}^r \mu_i \qquad \mbox{for all $r\geq 1$}.$$

 We write $\lambda \rhd \mu$ if   $\lambda \trianglerighteq  \mu$  and  $\lambda \neq  \mu$. 
For $\pointla , \pointmu \in \mathcal{Y}_k$, we write $\pointla  \vartriangleright \pointmu$ whenever $\lambda \vartriangleright \mu$.

We extend this to a reverse lexicographic ordering on the paths in $\Std_k(\lambda)$ as follows.  Given $\sts , \stt \in \Std_k(\lambda)$, we write $s \succ t$ if there exists $0 \leq  r < k$ such that $\sts (i)=\stt (i)$ for all $r <  i\leq  k$ and $\sts (r)\rhd \stt (r)$.
We write $\sts \succeq \stt $ if we have either $\sts \succ \stt $ or $\sts =\stt $.

\end{definition}

Working inductively along the edges of the branching graph, Goodman and Enyang  constructed  in \cite[Section~6]{EG:2012} an element $m_{\sts\stt}\in P_{k}^{R}(z)$ 
 associated to any pair of paths  $\sts, \stt \in \Std_k(\lambda)$. They showed that the set of all such elements form a cellular basis for $P_k^R(z)$. More precisely we have the following result.

 \begin{theorem}\cite[Proposition 6.26 and Theorem 6.30]{EG:2012}
  The algebra   $\algebra{k}{R}(z)$  is free as an $R$-module with basis  $$\left\{
m_{\sts\stt} \mid 
 \sts, \stt \in \Std_k(   \lambda),  
\obpointla\in\mathcal{Y}_k
\right\} .$$ 
Moreover, if $\sts,\stt \in \Std_k(\lambda)$, for some
      $\obpointla \in\mathcal{Y}_k$, and $a\in  \algebra{k}{R}(z) $ then 
    there exist scalars $r_{\stt\stu}(a)\in R$, which do not depend on
    $\sts$, such that 
\begin{equation}\label{constants}
 m _{\sts\stt} a =\sum_{\stu\in
      \Std_k(\lambda)}r_{\stt\stu}(a)m_{\sts\stu}\pmod 
      { \subalgebra (z)	},
      \end{equation}
      where $\subalgebra(z)$ is the $R$-submodule of $\algebra{k}{R}(z)$ spanned by
      \[\{m_{\sf q r}\mid\mu \vartriangleright  \lambda\text{ and }{\sf q , r}\in \Std_k(\mu )\}.\]
Finally,  we have that 
      $(m_{\sts\stt})^*=m_{\stt\sts}$, for all $(\lambda,k)\in\mathcal{Y}_k$ and
      all $\sts,\stt\in\Std_k(\lambda)$.
 Therefore the algebra $P_k^R(z)$  is cellular, in the sense of \cite{MR1376244}.  
 \end{theorem}

\begin{definition}\label{cellmodule}
Given  $\obpointla \in\mathcal{Y}_k$, the {\sf cell module} 
$\standard{k,R}{z}(\lambda)$ 
is the right $\algebra{k}{R}(z)$-module  with   linear basis $\big\lbrace m_\stt \mid\stt \in\Std_k(\lambda)\big\rbrace$ and right $\algebra{k}{R}(z)$-action
\begin{align*}
m_\stt a = \sum_{\stu \in \Std_k(\lambda) }r_{\stt \stu}(a) m_\stu , \quad\text{for $a\in  \algebra{k}{R}(z)$,}
\end{align*}
where the sum is over all $\stu \in\Std_k(\lambda)$ and the  $r_{\stt\stu}(a) \in R$ are determined by \cref{constants}. 
\end{definition}

These cellular basis are compatible with the restriction from $P_k^R(z)$ to $P_{k-1}^R(z)$. More precisely we have the following result.

\begin{proposition}[{\cite[Lemma~3.12]{EG:2012}}]\label{filtration}
Let $\obpointla \in\mathcal{Y}_k$ and assume that  $(\rho,k-1)\to\pointla $ is an edge in  $\mathcal{Y}$. Define $N^{\rhd \rho} \subseteq N^{\unrhd \rho}\subseteq \standard{k,R}{z}(\lambda)$  by 
\begin{eqnarray*}
&&N^{\unrhd \rho}= R {\rm -span} \{ m_\stt \in \Delta_{k,R}^z( \lambda) \mid  \stt  ({k-1})\unrhd(\rho,k-1)\} 
\quad\text{and}\\
&&N^{\rhd \rho}= R {\rm -span}\{ m_\stt \in\Delta_{k,R}^z( \lambda)  \mid  \stt  ({k-1})\rhd(\rho,k-1)\}.
\end{eqnarray*}
Then $N^{\rhd \rho}$ and $N^{\unrhd \rho}$ are $P_{k-1}^R(z)$-submodules of ${\rm res}^k_{k-1} \Delta_{k,R}^z(\lambda)$ and the linear map \newline $N^{\unrhd \rho}/N^{\rhd \rho} \to \standard{k-1,R}{z}(\rho)$ given by 
\begin{align*}
m_\stt +N^{\rhd \rho}  \mapsto  m_{\stt {\downarrow}_{k-1}}, \quad\text{for $\stt\in \Std_k(\lambda)$ with $\stt (k-1) = (\rho, k-1)$}
\end{align*}
is an isomorphism of $\algebra{k-1}{R}(z)$-modules.
\end{proposition}

The cellular structure on $P_k^R(z)$ gives a bilinear form on the cell modules defined as follows.

\begin{definition}\label{bilinearform}
 Let $\obpointla \in\mathcal{Y}_k$ and $\sts,\stt\in\Std_k(\lambda)$. \\  We define a  map $\langle \, , \, \rangle :\standard{k,R}{z}(\lambda)\times\standard{k,R}{z}(\lambda)\to R$ by 
\begin{align*}
m_{\sts \stu}  m_{\stv\stt} = 
\langle m_\stu , m_\stv  \rangle
m_{\sts\stt}   
 \mod \algebra{\vartriangleright(\lambda,k)}{R}(z),\qquad\text{for $\stu ,\stv \in\Std_k(\lambda)$.}
\end{align*}
Then $\langle \, , \, \rangle$ is a symmetric bilinear form on $\standard{k,R}{z}(\lambda)$ which is independent of the choice of $\sts,\stt \in \Std_k(\lambda)$.  
\end{definition}

\section{Jucys--Murphy elements and seminormal basis}

In \cite[Section~3]{MR2143201} a family of so-called {\sf Jucys--Murphy}  elements $L_1,\ldots,L_k$  for $\algebra{k}{R}(z)$ is  defined diagrammatically.   A recursion for these Jucys--Murphy elements, analogous to the recursion for the  Jucys--Murphy elements in the group algebra of the symmetric group, is given in~\cite[Section~3]{MR3035512} or ~\cite[Proposition~2.6]{MR3092697}. We will not need the explicit definition of these elements here but will recall their properties.

\begin{proposition}[\mbox{See~\cite[Proposition~3.15]{MR3092697}}]\label{u-t-c}
Assume that $\obpointla \in\mathcal{Y}_k$. If $\stt \in\Std_k(\lambda)$ and $1\le i\le k$, then 
\begin{align*}
m_\stt L_i=c_\stt (i)m_\stt  +\sum_{\sts \succ \stt }r_\sts m_\sts 
\end{align*}
for scalars $r_\sts , c_\stt (i) \in R$. Moreover, if we write $\stt = (\stt(0), \stt(1), \stt(2), \ldots , \stt(k))$, the coefficients $c_{\stt}(i)\in R$ are given as follows.
 If $i$ is even, we have  
\begin{align*}
c_\stt (i)=
\begin{cases}
z-|\stt(i)|,&\text{if $\stt(i)=\stt{(i-1)}$,}\\
c(a),&\text{if $\stt{(i)}=\stt{(i-1)}\cup\lbrace a\rbrace$,}
\end{cases}
\end{align*}
and, if $i$ is odd, we have
\begin{align*}
c_\stt (i)=
\begin{cases}
|\stt(i)|, &\text{if $\stt{(i)}=\stt{(i-1)}$,}\\
z-c(a),&\text{if $\stt{(i)}=\stt{(i-1)}\setminus \lbrace a\rbrace$.}
\end{cases}
\end{align*} 
\end{proposition}

\noindent Therefore the elements 
$L_1, \ldots, L_k$ form a family of Jucys--Murphy elements as in  \cite{MR2414949}.

\begin{definition}
Given $\stt \in\Std_k(\lambda)$, we let  $c(\stt)$ denote the {\sf content vector} $(c_\stt(1),c_\stt(2), \ldots, c_\stt(k))$.
\end{definition}

\begin{example}\label{exampleofcontent}
For   $k=6$, we  let $\stt \in \Std_6(\varnothing)$ and $\sts \in \Std_6((3))$ denote the   paths 
$$
 (\varnothing,\varnothing, (1),(1),(1), \varnothing,\varnothing)
\quad
 (\varnothing,\varnothing, (1),(1),(2), (2), (3))
$$ respectively.  These paths 
have  content vectors  $c(\stt)$ and $c(\sts)$ given by 
$$  
(0,0,1,z-1,z,z) \quad 
 (0,0,1,1,2,2)$$ 
 respectively.  
 \end{example}

A straighforward induction on $k$ shows that this family of Jucys--Murphy elements satisfies the separation condition (over $R$) given in  \cite[Definition 2.8]{MR2414949}, which in essence says that the content vectors distinguish between the elements of ${\rm Std}_k$. This allows us to develop further the representation theory of the partition algebra over the field of fractions $\mathbb{F}=\mathbb{Q}(z)$.

We  write 
\begin{align*}
\algebra{k}{\FF}(z)=\algebra{k}{R}(z)\otimes_R \mathbb{F}.
\end{align*}
To simplify notation, we will freely write $L_i,m_ {\stt},$ and so on, in place of $L_i\otimes 1_\mathbb{F}$,  $m_ {\stt}\otimes 1_\mathbb{F}$, and so on.    If $\obpointla \in\mathcal{Y}_k$, we  define 
\begin{align*}
  \standard{k,\mathbb{F}}{z}(\lambda)= \standard{k,R}{z}(\lambda)\otimes_R \mathbb{F}.
\end{align*}
The algebra $P_k^{\mathbb{F}}(z)$ is semisimple and the set $\{  \standard{k,\mathbb{F}}{z}(\lambda) \,\, : \,\, (\lambda ,k)\in \mathcal{Y}_k\}$ form a complete set of isomorphism classes of simple $P_k^{\mathbb{F}}(z)$-modules.

Following  \cite{MR2414949}, Enyang constructed in \cite{MR3092697} seminormal bases for the partition algebra over the field $\mathbb{F}$ which diagonalise the action of the Jucys--Murphy elements and are orthogonal with respect to the bilinear form $\langle \, , \, \rangle$. We recall the construction of these bases and some of their properties.
\begin{definition}\label{basis:1}
Let $\obpointla \in\mathcal{Y}_k$ and   $\stt\in\Std_k(\lambda)$. Define
\begin{align*}
F_\stt =
\prod_{
1\le i\le k
}
\prod_{
\substack{
\stu \in    \Std_k(\rho) 		\\
c_\stu (i)\ne c_\stt (i)}}
\frac{L_i-c_\stu (i)}{c_\stt (i)-c_\stu (i)},
\end{align*}
where the product is taken over all $(\rho,k)\in\mathcal{Y}_k$. We also define $$f_\stt =m_\stt F_\stt.$$
\end{definition}

\begin{proposition}[\mbox{See~\cite[Proposition~4.2]{MR3092697}}]\label{s-n-d}
Let $k\in \ZZ_{\geq 0}$ and $\obpointla \in\mathcal{Y}_{k}$. 
\begin{enumerate}[label=(\arabic{*}), ref=\arabic{*},leftmargin=0pt,itemindent=1.5em]
\item\label{s-n-d:1} If $\stt \in\Std_k(\lambda)$, then we have that
\begin{align*}
f_\stt =m_\stt 
+\sum_{\begin{subarray}c
\sts \in\Std_k(\lambda) \\
{\sts \succ\stt }
\end{subarray}
} r_\sts m_\sts,
\end{align*}
for scalars  $r_\sts \in \mathbb{F}$.
\item\label{s-n-d:2} The set  ${\lbrace}f_\stt \mid \stt \in\Std_k(\lambda){\rbrace}$ is an $\mathbb{F}$-basis for $\standard{k,\mathbb{F}}z(\lambda)$. 
\item\label{s-n-d:4} We have $f_\stt L_i=c_\stt (i)f_\stt $ for all $\stt \in\Std_k(\lambda)$ and $i=1,\ldots,k$. 
\item\label{s-n-d:5} We have  $F_\sts F_\stt  =\delta_{\sts\stt}F_\sts $ and $f_\sts F_{\stt }= \delta_{\sts \stt }f_\sts $ for all $\sts ,\stt \in\Std_k(\lambda)$.
\item\label{s-n-d:6} We have  $\langle f_\sts ,f_\stt \rangle=\delta_{\sts \stt}  \langle f_\sts ,f_\sts \rangle$ for all $\sts ,\stt \in\Std_k(\lambda)$.
\end{enumerate}

\end{proposition}
 
Enyang described explicitly the action of the generators of the partition algebra on the seminormal basis ${\lbrace}f_\stt \mid \stt \in\Std_k(\lambda){\rbrace}$ for $\standard{k,\mathbb{F}}z(\lambda)$. Enyang also gave an explicit inductive formula for the value of the bilinear form on the seminormal basis elements. Here we only recall what we will need later in the paper.

\begin{proposition}[\mbox{See~\cite[Proposition~4.12]{MR3092697}}]\label{branching}
Let $(\lambda, k)\in \mathcal{Y}_k$ and $\stt \in {\rm Std}_k(\lambda)$ with $\stt (k-1)=\eta$. Write $\sts = \stt{\downarrow}_{k-1}$. Then we have
\begin{align*}
\langle f_\stt ,f_\stt \rangle  =\gamma_{(\eta, k-1) \to(\lambda,k)} \langle f_{\sts} ,f_{\sts} \rangle
\end{align*}
where $\gamma_{(\eta, k-1) \to (\lambda ,k)}$ satisfies the following:
\begin{enumerate}
\item If $k$ is odd and $\lambda = \eta$, or $k$ is even and $\lambda = \eta \cup \{\alpha\}$ then $\gamma_{(\eta, k-1)\rightarrow (\lambda , k)} \in \mathbb{Q}$.
\item If $k$ is odd and $\lambda = \eta \setminus \{\alpha\}$ then
$$\gamma_{(\eta, k-1)\rightarrow (\lambda , k)} = \frac{(z-c(\alpha)-|\lambda| -1)}{(z-c(\alpha) -|\lambda|)}\,  r$$
for some $r\in \mathbb{Q}$.
\item If $k$ is even and $\lambda = \eta$ then
$$\gamma_{(\eta, k-1)\rightarrow (\lambda , k)} = \frac{\prod_{\beta\in A(\lambda)}(z-c(\beta) -|\lambda|)}{\prod_{\beta\in R(\lambda)}(z-c(\beta)-|\lambda|)}\,  r'$$
for some $r'\in \mathbb{Q}$.
\end{enumerate}
\end{proposition}

We will also make use of the following result

\begin{proposition}[\mbox{See~\cite[Theorem~3.16]{MR2414949}}]\label{ssidempotents} For each $(\lambda, k)\in \mathcal{Y}_k$ define
$$F_{(\lambda, k)} = \sum_{\stt\in {\rm Std}_k(\lambda)} F_\stt.$$
Then the set $\{F_{(\lambda, k)} \mid (\lambda ,k)\in \mathcal{Y}_k\}$ forms a complete set of pairwise orthogonal primitive central idempotents in $P_k^{\mathbb{F}}(z)$ and we have
$$\Delta_{k,\mathbb{F}}^z(\mu) F_{(\lambda, k)} = \delta_{\lambda \mu} \Delta_{k,\mathbb{F}}^z(\mu)$$
for all $(\mu,k)\in \mathcal{Y}_k$.
\end{proposition}

\section{Generalising seminormal basis to the non-separated case}
 
Throughout this section we fix  $n\in \mathbb{Z}_{\geq 0}$.

We wish to study the representations of the partition algebra over the field $\mathbb{Q}$ and with parameter $n$. We will relate these to the representations over $R$ and $\mathbb{F}$ in the usual manner. Given $n\in \mathbb{Z}_{\geq 0}$, we define $\mathcal{O}_n$ to be the localisation of $R$ at the prime ideal $\mathfrak{p}=(z-n)$. Then we have a natural embedding  $R\hookrightarrow  \mathcal{O}_n \hookrightarrow \mathbb{F}$ and  projection map $\pi_n : \mathcal{O}_n \rightarrow \mathbb{Q}$ given by specialising to $z=n$. We can consider the $\mathcal{O}_n$- and $\mathbb{Q}$-algebras
$$P_k^{\mathcal{O}_n}(z)=P_k^R(z) \otimes_R \mathcal{O}_n \quad \mbox{and} \quad P_k^{\mathbb{Q}}(n)=P_k^{\mathcal{O}_n}(z) \otimes_{\mathcal{O}_n} \mathbb{Q},$$
and their cell modules
$$\Delta_{k,\mathcal{O}_n}^{z}(\lambda) = \Delta_{k,R}^z(\lambda)\otimes_R \mathcal{O}_n \quad \mbox{and} \quad \Delta_{k,\mathbb{Q}}^n(\lambda) = \Delta_{k,\mathcal{O}_n}^{z}(\lambda) \otimes_{\mathcal{O}_n} \mathbb{Q}$$
for $(\lambda, k)\in \mathcal{Y}_k$. To simplify notation, we will freely write $L_i$ and  $m_{\stt}$ instead of $L_i\otimes 1_{\mathbb{Q}}$ and $m_\stt \otimes 1_\mathbb{Q}$. 

Over the field $\mathbb{Q}$ the family of Jucys--Murphy elements $L_i$, $1 \leq i \leq k$  no longer satisfies the separation property, and so the algebra $P_k^\mathbb{Q} (n)$ is not semisimple in general. 

The radical, $\rad \Delta_{k,\mathbb{Q}}^n(\lambda)$, of the form $\langle\,  ,\,  \rangle$ on the cell modules $\Delta_{k,\mathbb{Q}}^n(\lambda)$  is non-trivial in general. We can define
$$L_{k,\mathbb{Q}}^n(\lambda) = \Delta_{k,\mathbb{Q}}^n(\lambda) / \rad \Delta_{k,\mathbb{Q}}^n(\lambda).$$
Then, by the general theory of cellular algebra developed by Graham and Lehrer in \cite{MR1376244} we have that 
$$\{ L_{k,\mathbb{Q}}^n(\lambda) \neq 0 \mid (\lambda ,k)\in \mathcal{Y}_k\}$$
form a complete   set of pairwise  non-isomorphic simple $P_k^{\mathbb{Q}}(n)$-modules (see \cite[Theorem 3.4]{MR1376244}). Moreover, we have that if $L_{k,\mathbb{Q}}^n(\mu)$ is a composition factor of $\Delta_{k,\mathbb{Q}}^n(\lambda)$ then $\mu\unlhd \lambda$ (see \cite[Proposition 3.6]{MR1376244}).

Any algebra  $A$ decomposes into a direct sum of indecomposable two-sided ideals, called the {\sf blocks} of the algebra $A$. It is a general fact that every simple $A$-module is a composition factor of a unique block of $A$. Moreover, if $A$ is a cellular algebra, then it is  known that all composition factors of a cell module belong to the same block of $A$ (see \cite[3.9.8]{MR1376244}).

The partition algebra $P_k^\mathbb{F}(z)$ studied in Section 2  is semisimple and so its block decomposition is simply given by the two-sided ideals generated by the primitive central idempotents given in Proposition \ref{ssidempotents}.

Mathas developed a general framework in \cite[Section~4]{MR2414949} which (partially) generalises the theory of seminormal bases and associated central idempotents to the non-separated case. We now recall his results in the case of the partition algebra $P_k^{\mathbb{Q}}(n)$.

\begin{definition}[\mbox{See~\cite[Definition~4.1]{MR2414949}}] \label{residueeee}
\ 

\begin{enumerate}
\item For  $\stt \in\Std_k$, we define $r_{n, \stt}(i) = \pi_n(c_{\stt}(i))$ and  the $n$-{\sf residue vector}, $r_n(\stt)$, to be the vector $(r_{n,\stt}(1), r_{n,\stt}(2), \ldots , r_{n,\stt}(k))$, that is, the specialisation of the content vector $c(\stt)$ at $z=n$.
\item For $  \sts, \stt \in \Std_k$ we say that $\stt$ and $\sts$ are in the same $n$-{\sf residue class} and write
$\stt \approx_n \sts$ if $r_n(\stt )=r_n(\sts).$ 

\item For $(\lambda, k)\in \mathcal{Y}_k$ and $\stt\in \Std_k(\lambda)$ we write
$$
[\stt]_n = \{ \sts\in \Std_k \mid  \sts \approx_n \stt\} \quad \mbox{and} \quad
[\stt]_n^{(\lambda, k)} = \{ \sts \in \Std_k(\lambda) \mid \sts \approx_n \stt \}.
$$
\item Let $(\lambda, k), (\mu,k)\in \mathcal{Y}_k$. We write $(\lambda, k) \sim_n (\mu,k)$ if there exists $\stt_0, \stt_1, \ldots , \stt_r\in \Std_k$ with $\stt_0\in \Std_k(\lambda)$ and $\stt_r\in \Std_k(\mu)$ such that $\stt_j \approx_n \stt_{j+1}$ for all $j=0, 1, \ldots , r-1$. In this case we say that $(\lambda,k)$ and $(\mu,k)$ are in the same $n$-linkage class.
\end{enumerate}

\end{definition}

\begin{example}
Let $n=2$ and $k=6$ as in \hyperref[piccy]{Figure~\ref*{piccy}}. 
We have two paths 
$$
   \stt=(\varnothing,\varnothing, (1),(1),(1), \varnothing,\varnothing)
\quad
\sts=(\varnothing,\varnothing, (1),(1),(2), (2), (3))
$$
whose content vectors are given in Example \hyperref[exampleofcontent]{\ref*{exampleofcontent}}.  Specialising to $z=2$ we have that $r_2(\stt)=r_2(\sts)=(0,0,1,1,2,2)$ and so $\stt \approx_2 \sts$ and $(\varnothing, 6)\sim_2 ((3),6)$.
 \end{example}

\begin{definition} 
Let $(\lambda, k)\in \mathcal{Y}_k$ and $\stt\in \Std_k(\lambda)$. We define
$$F_{[\stt]_n} = \sum_{\sts\in [\stt]_n} F_{\sts} \in P_k^{\mathbb{F}}(z) \quad \mbox{and} \quad \tilde{f}_{\stt ,n} = m_{\stt} F_{[\stt]_n}\in \Delta_{k,\mathbb{F}}^z(\lambda).$$
\end{definition}

The next lemma shows that in fact these elements are defined over the ring $\mathcal{O}_n$.

\begin{lemma}[\mbox{See~\cite[Lemma~4.2]{MR2414949}}]
Let $(\lambda, k)\in \mathcal{Y}_k$ and $\stt\in \Std_k(\lambda)$. Then we have
$$F_{[\stt]_n}  \in P_k^{\mathcal{O}_n}(z) \quad \mbox{and} \quad \tilde{f}_{\stt  ,n} \in \Delta_{k,\mathcal{O}_n}^z(\lambda).$$
\end{lemma}

We can therefore make the following definition.

\begin{definition}
For each $(\lambda, k)\in \mathcal{Y}_k$ and $\stt\in \Std_k(\lambda)$ we define
$$g_\stt = \tilde{f}_{\stt ,n} \otimes 1_{\mathbb{Q}} \in \Delta_{k,\mathbb{Q}}^n(\lambda).$$
\end{definition}

\begin{proposition}[\mbox{See~\cite[Proposition~4.9]{MR2414949}}]\label{innerproductg}
Let $(\lambda, k)\in \mathcal{Y}_k$. The set $\{g_{\stt} \, : \, \stt\in \Std_k(\lambda)\}$ form a basis for $\Delta_{k,\mathbb{Q}}^n(\lambda)$. Moreover, for $  \sts, \stt \in \Std_k(\lambda)$ we have
$$\langle g_{\sts}, g_{\stt} \rangle = 
\left\{\begin{array}{ll}\langle m_{\sts} , g_{\stt} \rangle & \mbox{if $\sts \approx_n \stt$}
\\
0 & \mbox{otherwise} \end{array} \right.$$
\end{proposition}

\begin{definition}
For each $(\lambda, k)\in \mathcal{Y}_k$, define $[(\lambda ,k)]_n = \{ \sts\in \Std_k \mid  \sts\approx_n \stt \,\, \mbox{for some $\stt \in \Std_k(\lambda)$}\}$ and set
$$F_{[(\lambda , k)]_n} = \sum_{\sts\in [(\lambda,k)]_n} F_\sts.$$
Then $F_{[(\lambda ,k)]_n}\in P_k^{\mathcal{O}_n}(z)$ and so we can define
$$G_{(\lambda,k)} = F_{[(\lambda,k)]_n}\otimes 1_{\mathbb{Q}} \in P_k^\mathbb{Q}(n).$$
\end{definition}

\noindent Note that if $(\mu,k)\sim_n (\lambda,k)$ then $G_{(\lambda,k)}=G_{(\mu,k)}$.

\noindent Let $\mathcal{Y}_k / \sim_n$ be a set of representatives for the $n$-linkage classes on $\mathcal{Y}_k$.

\begin{proposition}[\mbox{See~\cite[Corollary~4.6~and~4.7]{MR2414949}}] \label{mathasnecblock}
The set $\{G_{(\lambda, k)} \mid (\lambda, k)\in \mathcal{Y}_k / \sim_n\}$ form a complete set of pairwise orthogonal central idempotents in $P_k^{\mathbb{Q}}(n)$ and we have
$$\Delta_{k, \mathbb{Q}}^n(\mu) G_{(\lambda , k)} = \left\{ \begin{array}{ll} \Delta_{k,\mathbb{Q}}^n(\mu) & \mbox{if $(\mu,k)\sim_n(\lambda,k)$}\\ 0 & \mbox{otherwise}\end{array}\right.$$
In particular, if $\Delta_{k,\mathbb{Q}}^n(\lambda)$ and $\Delta_{k,\mathbb{Q}}^n(\mu)$ are in the same block then $(\lambda,k)\sim_n(\mu,k)$.
\end{proposition}

We will see in the next section that in fact the $G_{(\lambda,k)}$'s are primitive central idempotents; equivalently, that  $\Delta_{k,\mathbb{Q}}^n(\lambda)$ and $\Delta_{k,\mathbb{Q}}^n(\mu)$ are in the same block if and only if $(\lambda,k)\sim_n(\mu,k)$.

\section{Residue classes, Reflection geometry and blocks}

In this section we give a geometrical interpretation of the $n$-linkage classes on $\mathcal{Y}_k$ and the $n$-residue classes on $\Std_k$. The main idea was already introduced in \cite{BDK} but here we take it further by looking at the geometry on the whole branching graph $\mathcal{Y}$ (rather than just on one level $\mathcal{Y}_k$).

Let $\{\varepsilon_0,\varepsilon_1, \varepsilon_2, \dots\}$ be a set of formal symbols and set
	\[\mathbb{Z}^\infty=\prod_{i\geq 0}\mathbb{Z}\varepsilon_i\]
We will write each $x=\sum_{i\geq 0} x_i \varepsilon_i\in \mathbb{Z}^\infty$ as a vector 
$x=(x_0,x_1,x_2, \ldots).$
We take the infinite symmetric group $\mathfrak{S}_{\infty}$ to be the group generated by the transpositions $s_{i,j}$ (for $i,j\geq 0$ and $i\neq j$) where each $s_{i,j}$ acts on $\mathbb{Z}^\infty$ by permuting the $i$-th  and $j$-th coordinates of the vectors.

Define the graph $\mathcal{Z}$ with vertex set $\cup_{k\geq 0}\mathcal{Z}_k$ where $\mathcal{Z}_k=\mathbb{Z}^{\infty}$ for all $k\geq 0$ and 
\begin{enumerate}
\item if $k$ is even and $x\in \mathcal{Z}_k$, an edge $x\rightarrow y$ if $y\in  \mathcal{Z}_{k+1}$ with $y=x-\varepsilon_i$ for some $i\geq 0$,
\item if $k$ is odd and $x\in \mathcal{Z}_k$, an edge $x\rightarrow y$ if $y\in  \mathcal{Z}_{k+1}$ with $y=x+\varepsilon_i$ for some $i\geq 0$.
\end{enumerate}

  For any partition $\lambda=(\lambda_1, \lambda_2, \lambda_3,\ldots )$ define
$$\lambda_{[n]} = (n-|\lambda|, \lambda_1, \lambda_2, \lambda_3,\ldots )\in \mathbb{Z}^{\infty}$$
(adding infinitely many zeros after the last part of $\lambda$).
Define also $\rho\in \mathbb{Z}^\infty$ by 
$$\rho = (0, -1,-2,-3, \ldots).$$
We define an embedding $\varphi_n$ of the graph $\mathcal{Y}$ into $\mathcal{Z}$ as follows. 
For each $k\geq 0$ and each $(\lambda, k)\in \mathcal{Y}_k$ we set
\begin{equation}\label{embedding}
\varphi_n(\lambda, k) = \left\{ \begin{array}{ll} \lambda_{[n]} +\rho & \mbox{if  $k$ is even,}\\
 \lambda_{[n-1]}+\rho & \mbox{if  $k$ is odd.}
\end{array}\right.
\end{equation}
Note that $\varphi_n(\mathcal{Y})$ is then the full subgraph of $\mathcal{Z}$ on the vertex set $\varphi_n(\cup_{k\geq 0}\mathcal{Y}_k)$.

\begin{remark}\label{wallalcove}
Note that if $x=(x_0,x_1,x_2,x_3, \ldots)$ is a vertex in $\varphi_n(\mathcal{Y})$ then we have
$$x_1>x_2>x_3> \ldots $$
Moreover we either have $x_0>x_1$, or $x_0=x_j$ for some $j\geq 1$, or $x_{j-1}>x_0>x_{j}$ for some $j>1$.
\end{remark}

The composition factors of the cell modules $\Delta_{k,\mathbb{Q}}^n(\lambda)$ for all $(\lambda, k)\in \mathcal{Y}_k$ were originally described (for $k$ even and $n\neq 0$) by P. Martin in \cite{mar1} in terms of $n$-pairs of partitions. This result was then extended to include the case $n=0$ by Doran and Wales in \cite{DW}.   We now reformulate it in terms of reflections on the graph $\mathcal{Z}$.  This reformulation had already been observed in \cite{BDK} in a slightly different form.

\begin{definition}
Let $\lambda,\mu$ be partitions. We say that $(\lambda,\mu)$ from an $n$-pair if $\lambda \subset \mu$ and $\mu$ differs from $\lambda$ by a single row of boxes the last of which having content $n-|\lambda|$.
\end{definition}

\begin{lemma}\label{npairreflection}
Let $(\lambda,k), (\mu,k)\in \mathcal{Y}_k$ with $|\lambda|<|\mu|$. Then we have
$\varphi_n(\mu,k) = s_{0,j}(\varphi_n(\lambda,k))$ for some $j\geq 1$
if and only if either $k$ is even and $(\lambda,\mu)$ form an $n$-pair, or $k$ is odd and $(\lambda,\mu)$ form an $(n-1)$-pair. 
\end{lemma}

\begin{proof} (See also \cite[proof of Theorem 6.4]{BDK}). We prove the result for $k$ even. The case $k$ odd is identical.
Suppose that $(\lambda, \mu)$ form an $n$-pair. Then by definition, there exists some $j\geq 1$ such that
$$\mu = (\lambda_1, \lambda_2 , \ldots , \lambda_{j-1}, \mu_j, \lambda_{j+1}, \ldots )$$
with $\mu_j - j = n-|\lambda|$. So we have
$$
\varphi_n(\mu,k) = (n-|\mu|, \lambda_1 - 1, \ldots , \lambda_{j-1}-(j-1), n-|\lambda|, \lambda_{j+1}-(j+1), \ldots ).
$$
Now, as $|\mu| = |\lambda| + (\mu_j - \lambda_j)$ and $\mu_j = n-|\lambda| + j$  we have
\begin{eqnarray*}
n-|\mu| &=& n-(|\lambda|+(\mu_j - \lambda_j))\\
&=& n-|\lambda| - (n-|\lambda| +j) + \lambda_j\\
&=& \lambda_j -j.
\end{eqnarray*}
Thus we have $\varphi_n(\mu, k) = s_{0,j}(\varphi_n(\lambda,k))$ as required.

Conversely, suppose that $\varphi_n(\mu, k) = s_{0,j}(\varphi_n(\lambda,k))$ for some $j\geq 1$. Then $\lambda$ and $\mu$ only differ in row $j$ and the last box in row $j$ of $\mu$ has content $\mu_j - j = n-|\lambda|$. So $(\lambda,\mu)$ form an $n$-pair.
\end{proof}

Using Lemma \ref{npairreflection}, we can now reformulate Martin's (and Doran and Wales') result about the blocks and decomposition numbers for the partition algebra and extend it to include the case when $k$ is odd.

\begin{theorem}\label{cfcells} Each block of $P_k^\mathbb{Q}(n)$ is given by a maximal chain
$(\lambda^{(1)}, k), (\lambda^{(2)}, k), \ldots , (\lambda^{(r)},k)$
for some $r\geq 1$ satisfying
$|\lambda^{(1)}| < |\lambda^{(2)}| < \ldots < |\lambda^{(r)}|$
and 
$\varphi_n(\lambda^{(j+1)},k) = s_{0,j} (\varphi_n(\lambda^{(j)},k))$ for all $j=1, \ldots , r-1$.

Moreover, the cell module $\Delta_{k,\mathbb{Q}}^n(\lambda^{(j)})$ for $j=1, \ldots, r-1$, has two composition factors, namely $L_{k,\mathbb{Q}}^n(\lambda^{(j)})$ as its head and $L_{k,\mathbb{Q}}^n(\lambda^{(j+1)})$ as its socle, and $\Delta_{k,\mathbb{Q}}^n(\lambda^{(r)}) = L_{k,\mathbb{Q}}^n(\lambda^{(r)})$ is simple, except for the case when $k$ is even, $n=0$, and $\lambda = \varnothing$ where we have $\Delta_{k,\mathbb{Q}}^0(\varnothing) \cong L_{k,\mathbb{Q}}^0((1))$.
\end{theorem}

\begin{remark}\label{remarkwall}
Note in particular that if $\varphi_n(\lambda ,k) = s_{0,j}(\varphi_n(\lambda ,k))$ for some $j\geq 1$ then $(\lambda,k)$ is alone in its block and $\Delta_{k,\mathbb{Q}}^n(\lambda) = L_{k,\mathbb{Q}}^n(\lambda)$.
\end{remark}

\begin{proof}
For the case $k$ is even, this is just a reformulation of  \cite[Proposition 9]{mar1} (and \cite{DW} for the case $n=0$). The case $k$ is odd and $n>1$ is obtained using the Morita equivalence between $P_{2k+1}^\mathbb{Q}(n)$ and $P_{2k}^\mathbb{Q}(n-1)$ given in \cite[Section 3]{martin2000partition} (see also \cite[Theorem 5.2]{BDK} for a detailed proof). This equivalence is obtained using an idempotent $\xi\in P_{2k+1}^\mathbb{Q}(n)$. We have $\xi P_{2k+1}^\mathbb{Q}(n) \xi \cong P_{2k}^\mathbb{Q}(n-1)$ and under this isomorphism we have $\xi \Delta_{2k+1, \mathbb{Q}}^n(\lambda) \cong \Delta_{2k,\mathbb{Q}}^{n-1}(\lambda)$ for all $(\lambda, 2k+1)\in \mathcal{Y}_{2k+1}$. 

Now when $n=1$, an easy calculation shows that the form
 $\langle \, , \, \rangle$ on any $\Delta_{2k+1,\mathbb{Q}}^1(\lambda)$ for $\lambda\in \mathcal{Y}_{2k+1}$ is non-degenerate.
  In particular,  the form
 $\langle \, , \, \rangle$ on 
 $\Delta_{2k+1,\mathbb{Q}}^1(\varnothing)$ is non-degenerate 
and so $L_{2k+1, \mathbb{Q}}^1(\varnothing) \neq 0$.   
This shows that the above Morita equivalence cannot hold in this case (as the number of simple modules of these two algebras does not coincide). However, we still have an idempotent functor from the category of $P_{2k+1}^\mathbb{Q}(1)$-modules to the category of $\xi P_{2k+1}^\mathbb{Q}(1) \xi \cong P_{2k}^\mathbb{Q}(0)$-modules taking cell modules to the corresponding cell modules (by an identical argument to that used in the proof of   \cite[Theorem 5.2]{BDK}).
  Given  $(\lambda,2k+1) \in \mathcal{Y}_{2k+1}$ and $(\mu,2k+1) \in\mathcal{Y}_{2k+1} \setminus \{\varnothing\}$ we have that
\begin{align*}
[ \Delta_{2k+1,\mathbb{Q}}^1((\lambda,2k+1)):  L_{2k+1,\mathbb{Q}}^1((\mu,2k+1))] &=
[\xi\Delta_{2k+1,\mathbb{Q}}^1((\lambda,2k+1)): \xi L_{2k+1,\mathbb{Q}}^1((\mu,2k+1))] 
\\
&=
  [\Delta_{2k,\mathbb{Q}}^0((\lambda,2k)) : L_{2k,\mathbb{Q}}^0((\mu,2k))]
\end{align*}
  by \cite[(6.6b)Lemma]{Green}.  Finally, we observe that $(\varnothing,2k+1)$ is maximal
  in the dominance order  and so   
 $\Delta_{2k+1, \mathbb{Q}}^1(\varnothing) $ is the 
  unique cell-module in which $L_{2k+1, \mathbb{Q}}^1(\varnothing) $  appears as a composition factor (and it appears with multiplicity equal to 1  in this module, by  cellularity).   Therefore the decomposition numbers are as claimed.    The structure of the cell modules follows immediately (because there are only two composition factors and each cell module has a simple head).  
 \end{proof}

%
%
%

Motivated by Theorem \ref{cfcells}  and Remark \ref{wallalcove}, we make the following definition.

\begin{definition}
Let $x=(x_0,x_1,x_2, \ldots )$ be a vertex in $\mathcal{Z}$. We say that $x$ is
\begin{enumerate}
\item   on the $j$-th wall if $x_0=x_j$ for $j\geq 1$; 
\item   in the first alcove if
$x_0>x_1>x_2>x_3>\ldots$
\item  in the $j$-th alcove, for some $j>1$, if
$x_1>x_2>\ldots > x_{j-1} > x_0 > x_j > x_{j+1} > \ldots$ 
\end{enumerate}
\end{definition}

\noindent The following lemma follows directly from \hyperref[wallalcove]{Remark~\ref*{wallalcove}}.

\begin{lemma}
Let $x$ be a vertex in $\varphi_n(\mathcal{Y})$. Then there exists a unique $j\geq 1$ such that either $x$ is in the $j$-th alcove or $x$ is on the $j$-th wall.
\end{lemma}

\begin{lemma}\label{reflectionimpliesadjacent} Let $(\lambda, k), \pointmu \in \mathcal{Y}_k$ with $|\lambda|<|\mu|$. If $\varphi_n\pointla =s_{0,j}(\varphi_n\pointmu )$ for some $j\geq 1$ then $\varphi_n \pointla $ is in the $j$-th alcove and $\varphi_n\pointmu $ is in the $(j+1)$-th alcove.
\end{lemma} 

\begin{proof}
Write $x=\varphi_n(\lambda,k)$ and $y=\varphi_n(\mu,k)$. We have
$x=(x_0,x_1,x_2,\ldots, x_{j-1} , x_j,x_{j+1},  \ldots)$
and
$y=(x_j, x_1, x_2, \ldots, x_{j-1}, x_0, x_{j+1}, \ldots).$
By assumption, $|\lambda|<|\mu|$ and $x_i=y_i$ for all $i\geq 1$ except for $i=j$; therefore  we must have $x_j<y_j$.
Thus we have
$$x_{j-1}>x_0=y_j> x_j$$
and $x$ is in the $j$-th alcove. And we also have
$$y_j=x_0 > y_0=x_j > y_{j+1}=x_{j+1}$$
and $y$ is in the $(j+1)$-th alcove.
\end{proof}

\begin{proposition}\label{edgealcove}
Let $x,y$ be vertices in $\varphi_n(\mathcal{Y})$. Assume that $x$ is in the $j$-th alcove and that $y\rightarrow x$ is an edge in the graph $\mathcal{Z}$. Then we are in precisely one of the following cases:
\begin{enumerate}
\item $y$ is in the $j$-th alcove.
\item $y$ is on the $j$-th wall and either $y=x+\varepsilon_j$ or $y=x-\varepsilon_0$.
\item $y$ is on the $(j-1)$-th wall and either $y=x+\varepsilon_0$ or $y=x-\varepsilon_{j-1}$.
\end{enumerate}
\end{proposition}

\begin{proof}
As $y\rightarrow x$ is an edge in $\mathcal{Z}$ we have that $y=x\pm \varepsilon_k$ for some $k\geq 0$.
As $x$ is in the $j$-th alcove we have
$$x_1>x_2>\ldots >x_{j-1} > x_0 > x_j > x_{j+1} > \ldots$$
and so $y$ satisfies
$$y_1>y_2>\ldots >y_{j-1} \geq y_0 \geq y_j >y_{j+1}> \ldots$$
This implies that either $y$ is in the $j$-th alcove, or $y_{j-1}=y_0$ (that is, $y$ is on the $(j-1)$-wall), or $y_j=y_0$ (that is, $y$ is on the $j$-th wall). Now we have $y_{j-1}=y_0$ precisely when either $y=x+\varepsilon_0$ or $y=x-\varepsilon_{j-1}$. Similarly we have $y_j=y_0$ precisely when either $y=x-\varepsilon_0$ or $y=x+\varepsilon_j$.
\end{proof}

\begin{lemma}\label{reflectionystuff}
Let   $\stt =(\stt(0), \stt(1), \ldots , \stt(k))\in \Std_k $ and let
$\varphi_n(\stt) =(x^{(0)}, x^{(1)}, \ldots , x^{(k)} )$
be its image in $\mathcal{Z}$.
Then the $n$-residue vector $r_n(\stt)$ is given by
$$r_{n, \stt }(i) = \left\{ \begin{array}{ll} x^{(i)}_j & \mbox{if $x^{(i)}=x^{(i-1)}+\varepsilon_j$ for some $j\geq 0$,}\\
n-1 -x^{(i)}_j & \mbox{if $x^{(i)}=x^{(i-1)}-\varepsilon_j$ for some $j\geq 0$.}\end{array}\right.$$
\end{lemma}

\begin{proof}
This follows directly from the definition of the embedding $\varphi_n$ given in (\ref{embedding}) and Definition \ref{residueeee}(1).

\end{proof}

 \begin{proposition}\label{intervals}
Let $\sts,\stt\in \Std_k$. Then we have $\sts \approx_n \stt $ if and only if $\sts $ and $\stt$ agree everywhere except possibly on a finite number of  intervals
$[a,b]$ for $0 < a < b  \leq   k$
 and for
  each such interval   $[a,b]$ there exists some $j\geq 1$ satisfying the following conditions.
\begin{enumerate}[label=(\arabic{*}), ref=\arabic{*},leftmargin=0pt,itemindent=1.5em]
\item  We have that $\stt (a)=\sts (a)$ and $\varphi_n(\stt (a))$ is on the $j$-th wall. Moreover,  if $b\neq k$ then we also have that $\stt (b)=\sts (b)$   and $\varphi_n(\stt(b))$  is on the $j$-th wall.
\item Either all $\varphi_n(\stt(i))$ for $a<i<b$ are in the $j$-th alcove and all $\varphi_n(\sts (i))$ for $a<i<b$ are in the $(j+1)$-th alcove, or vice versa.
\item  $\varphi_n(\stt (i))=s_{0,j}(\varphi_n(  \sts (i)) $ for all $a<i<b$.
\end{enumerate}
\end{proposition}

\begin{proof}
We write $$\varphi_n(\stt) =(x^{(0)}, x^{(1)}, x^{(2)}, \ldots, x^{(k)} ) \quad \varphi_n(\sts)  = (y^{(0)}, y^{(1)}, y^{(2)}, \ldots, y^{(k)}).$$  
Note that (2) follows from (3) using Lemma \ref{reflectionimpliesadjacent}. It is enough to show that 
  \begin{enumerate}[label=(\arabic{*}), ref=\arabic{*},leftmargin=0pt,itemindent=1.5em]
\item[(I)] if $x^{(i)} = y^{(i)}$ and $x^{(i+1)}\neq y^{(i+1)}$ then $x^{(i)}$ is on
 the $j$-th wall for some $j\geq 1$ and $y^{(i+1)} = s_{0,w}(x^{(i+1)})$; 
\item[(II)] if $x^{(i)}\neq y^{(i)}$ and $y^{(i)}=s_{0,j}  (x^{(i)})$ then $y^{(i+1)}=s_{0,j} (x^{(i+1)})$.
\end{enumerate} 
First we fix some notation.   For  $0\leq i \leq k$, we let $u_i$ (respectively  $v_i$)  denote the row  in which 
 the coordinates of $\varphi_n(\stt(i))$ and $\varphi_n(\stt(i-1))$      (respectively of 
 $\varphi_n(\sts(i))$ and $\varphi_n(\sts(i-1))$) differ, in other words
 $$x^{(i)} = x^{(i-1)}\pm \varepsilon_{u_i} \quad y^{(i)} = y^{(i-1)}\pm \varepsilon_{v_i}.$$
 It follows   from Lemma \ref{reflectionystuff} that 
$\stt \approx_n\sts $ if and only if 
 $$x^{(i)}_{u_i} = y^{(i)}_{v_i}$$  for all $1\leq i \leq k$.   
   We start by proving  (I). Assume that $x^{(i)}=y^{(i)}$ and $x^{(i+1)}\neq y^{(i+1)}$
   and so $u_i\neq v_i$.  
     As $\stt \approx_n\sts $ we must have
$$x^{(i+1)}_{u_i} = x^{(i)}_{u_i} \pm 1 = y^{(i+1)}_{v_i} = x^{(i)}_{v_i}\pm 1$$
and so $x^{(i)}_{u_i} = x^{(i)}_{v_i}$. Thus we must have that one of $u_i$ or $v_i$ is equal to $0$. Assume without loss of generality that ${v_i}=0$. Then we have
$$x^{(i)}_{u_i}=x^{(i)}_0$$
and so $x^{(i)}$ is on the $u_i$-th wall. To simplify the notation we set $u=u_i$. We have
$$y^{(i+1)}_m = x^{(i+1)}_m$$
for all $m$ except $m=0$ and $m=u$ and we have
$$y^{(i+1)}_0 = x^{(i+1)}_u.$$
We also have 
$$y_u^{(i+1)}=x^{(i)}_u = x^{(i)}_0 = x^{(i+1)}_0.$$
Thus we have
$$y^{(i+1)} = s_{0,u}(x^{(i+1)})$$
as required.   

 Now we turn to (II). Assume that $x^{(i)}\neq y^{(i)}$ and $y^{(i)}=s_{0,j}(x^{(i)})$. 
As $\stt \approx_n \sts$ we must have $x_{u_i}^{(i+1)} = y^{(i+1)}_{v_i}$ and so $x_{u_i}^{(i)} = y_{v_i}^{(i)}$. As $y^{(i)} = s_{0,j}(x^{(i)})$ we must have that one of $u_i$ or $v_i$ is equal to $0$ and the other is equal to $j$. Without loss of generality we assume that $u_i=0$ and $v_i=j$. Then we have
$$x^{(i+1)} = x^{(i)} \pm \varepsilon_0 \quad\quad y^{(i+1)} = y^{(i)} \pm \varepsilon_j.$$
Thus we have
$$x^{(i+1)} = (x_0^{(i)} \pm 1 , x_1^{(i)}, x_2^{(i)}, \ldots , x_j^{(i)}, \ldots ) \quad\quad y^{(i+1)} = (x_j^{(i)}, x_1^{(i)}, x_2^{(i)}, \ldots , x_0^{(i)}\pm 1, \ldots ).$$
So we have $y^{(i+1)} = s_{0,j}(x^{(i+1)})$ as required.
\end{proof}

\begin{example}
Let $k=6$ and $n=2$, the subgraph of the first 6 levels of $\varphi_2(\mathcal{Y})$ intersected with $\ZZ\{\varepsilon_0,\varepsilon_1,\varepsilon_2\}$ is depicted in \hyperref[piccy]{Figure~\ref*{piccy}} below.  The dashed lines on the diagram depict the $1$-st and $2$-nd walls; the vertices are drawn in such a manner that the points obtained by reflection through a wall matches the points one obtains by reflecting in the dashed lines of the diagram.  
 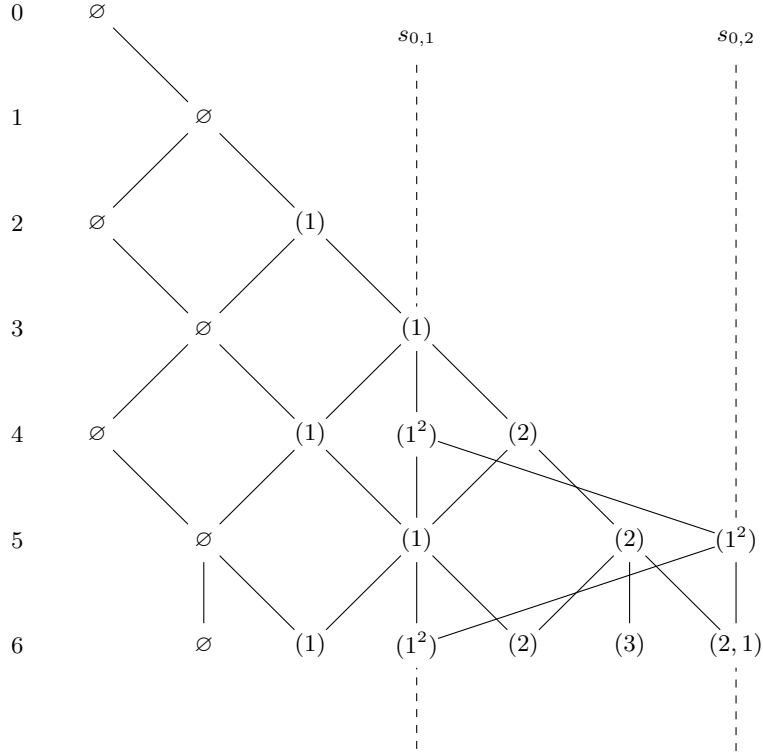
\begin{figure}[ht!]
$$  \scalefont{0.8}
\begin{tikzpicture}[scale=0.7]
          \begin{scope}  
             \draw (-3.5,2) node { 0 };  
 \draw (-3.5,0) node { 1 };  
 \draw (-3.5,-2) node { 2};  
 \draw (-3.5,-4) node { 3 };  
 \draw (-3.5,-6) node { 4 };  
 \draw (-3.5,-8) node { 5 };  
 \draw (-3.5,-10) node { 6 };  
             \draw[dashed](4,1)--(4,-4);
                          \draw[dashed](4,-10)--(4,-12);
                        \draw[dashed](10,1)--(10,-8);
                                                \draw[dashed](10,-12)--(10,-10);
               \draw (4,1.5) node { $s_{0,1}$ };            
    \draw (10,1.5) node { $s_{0,2}$ };            
          \draw(-2,2)--(0,0)--(-2,-2) --(0,-4)--(-2,-6)--(0,-8);
                    \draw (0,0)--(2,-2) --(0,-4)--(2,-6)--(0,-8);
 \draw (2,-2) --(8,-8);
  \draw (4,-4) --(2,-6)--(4,-8)--(4,-4);
    \draw (4,-6)--(10,-8);  \draw (6,-6)--(4,-8);
    \draw (0,-10)--(0,-8);
        \draw (2,-10)--(0,-8);
            \draw (2,-10)--(4,-8);
                \draw (4,-10)--(4,-8)--(6,-10);
                    \draw (4,-10)--(10,-8);                 
  \draw (6,-10)--(8,-8)--(8,-10);
     \draw (10,-8)--(10,-10)--(8,-8);
\fill[white] (-2,2) circle (12pt);   
              \fill[white] (0,0) circle (12pt);
    \fill[white] (-2,-2) circle (12pt);
        \fill[white] (+2,-2) circle (12pt);
      \fill[white] (0,-4)circle (12pt);
     \fill[white] (4,-4) circle (12pt);
      \fill[white] (+-2,-6) circle (12pt);
          \fill[white] (2,-6) circle (12pt);
             \fill[white] (+6,-6) circle (12pt);
                          \fill[white] (4,-6) circle (12pt);
                           \fill[white] (+-2,-10) circle (12pt);
          \fill[white] (2,-10) circle (12pt);
             \fill[white] (+6,-10) circle (12pt);
                          \fill[white] (4,-10) circle (12pt);
    \fill[white] (0,-8) circle (12pt);
          \fill[white] (4,-8) circle (12pt);
             \fill[white] (+8,-8) circle (12pt);
             \fill[white] (0,-10) circle (12pt);
\fill[white](10,-10)  circle (12pt);
                          \fill[white] (10,-8) circle (12pt);
 \fill[white] (8,-10) circle (12pt);         
           \draw (-2,2) node { $  \varnothing $  };   
              \draw (0,0) node {   $ \varnothing $  };   
    \draw (-2,-2) node   {   \text{	$ \varnothing  $	}}		;
        \draw (+2,-2) node   {   $ (1) $	    }		;
      \draw (0,-4) node   {   \text{	$ \varnothing$	}};
     \draw (4,-4) node   {  $ (1) $	 };
      \draw (+-2,-6) node   {   	  $ \varnothing $  };                 
          \draw (2,-6) node   {  $ {(1)} $	 }		;
             \draw (+6,-6) node              {    $  {(2)}  $	 	 }		;
                          \draw (4,-6) node            {    $ {(1^2)} $	 	 }			;
    \draw (0,-8) node   {   	  $ \varnothing $  };                 
          \draw (4,-8) node   {  $ {(1)} $	 }		;
             \draw (+8,-8) node              {    $  {(2)}  $	 	 }		;
                          \draw (10,-8) node            {    $ {(1^2)} $	 	 }			;        
                               \draw (0,-10) node   {   	  $ \varnothing $  };                 
          \draw (2,-10) node   {  $ {(1)} $	 }		;
             \draw (+6,-10) node              {    $  {(2)}  $	 	 }		;
                          \draw (4,-10) node            {    $ {(1^2)} $	 	 }			;                                        
                          \draw (8,-10) node            {    $ {(3)} $	 	 }			;                
\draw (10,-10) node            {    $ {(2,1)} $	 	 }			;                
    \end{scope}\end{tikzpicture}$$
    \caption{For $k=6$ and $n=2$, the subgraph of the first 6 levels of $\varphi_2(\mathcal{Y})$ intersected with $\ZZ\{\varepsilon_0,\varepsilon_1,\varepsilon_2\}$.}
    \label{piccy}
\end{figure}
\end{example}

It follows from Proposition \ref{intervals}  that if $(\lambda, k)\sim_n (\mu,k)$ with $\lambda \neq \mu$ then $\varphi_n(\lambda,k)$ must be in an alcove, say that $j_1$-th alcove and $\varphi_n(\mu,k)$ is in the $j_2$-th alcove say for some $j_1, j_2\geq 1$. We can assume that $j_2>j_1$. Moreover we must have 
$$\varphi_n(\mu,k) = s_{0,j_2-1}\ldots s_{0,j_1+1}s_{0,j_1}(\varphi_n(\lambda,k)).$$
But then using Theorem \ref{cfcells} we know that $\Delta_{k,\mathbb{Q}}^n(\lambda)$ and $\Delta_{k,\mathbb{Q}}^n(\mu)$ belong to the same block. Thus we can strengthen Proposition \ref{mathasnecblock} as follows.

\begin{theorem}\label{blocks} Let $(\lambda, k), (\mu,k)\in \mathcal{Y}_k$. Then we have that $\Delta_{k,\mathbb{Q}}^n(\lambda)$ and $\Delta_{k,\mathbb{Q}}^n(\mu)$ belong to the same block if and only if $(\lambda,k)\sim_n (\mu,k)$. In particular, we have that 
$$\{G_{(\lambda,k)} \, : \, (\lambda,k)\in \mathcal{Y}_k / \sim_n\}$$
form a complete set of primitive pairwise orthogonal central idempotents in $P_k^\mathbb{Q}(n)$.
\end{theorem}

\section{Restricting cell and simple modules}

We have seen in Proposition \ref{filtration} that the restriction of any cell module has a filtration by cell modules and that the factors appearing are determined by the branching graph.

In this section, we fix $n\in \mathbb{Z}_{\geq 0}$ and study in more details the restriction of cell modules $\Delta_{k,\mathbb{Q}}^n(\lambda)$ and simple modules $L_{k,\mathbb{Q}}^n(\lambda)$ when $\varphi_n(\lambda,k)$ is in an alcove.

\begin{theorem}\label{restrictioncells}
Let $(\lambda,k)\in \mathcal{Y}_k$ with $\varphi_n(\lambda,k)$ in an alcove. Then we have
$${\rm res}_{k-1}^k \Delta_{k,\mathbb{Q}}^n(\lambda) = \bigoplus_{(\eta,k-1)\rightarrow (\lambda,k)}\Delta_{k,\mathbb{Q}}^n(\lambda) G_{(\eta,k-1)}$$
where the direct sum is taken over all edges $(\eta,k-1)\rightarrow (\lambda,k)$ in $\mathcal{Y}$. Moreover,
$$\Delta_{k,\mathbb{Q}}^n(\lambda) G_{(\eta,k-1)} = \mathbb{Q}{\rm -span} \{ g_\stt \, : \, \stt \in \Std_k(\lambda) \, \mbox{with}\, \, \stt(k-1) = \eta\}$$
and we have an isomorphism
$$\Delta_{k,\mathbb{Q}}^n(\lambda) G_{(\eta,k-1)} \rightarrow \Delta_{k-1, \mathbb{Q}}^n(\eta) \, : \, 
g_\stt \mapsto g_{\stt{\downarrow}_{k-1}}.$$
\end{theorem}
 
In order to prove this theorem we will need the following four lemmas.

\begin{lemma}\label{restrictionlem1}
Let $\stt\in \Std_k$ and $\sts\in \Std_{k-1}$. Then we have
$$F_\stt F_\sts = \left\{ \begin{array}{ll} F_\stt & \mbox{if $\stt{\downarrow}_{k-1} = \sts$}\\ 0 & \mbox{otherwise.}\end{array}\right.$$
\end{lemma}

\begin{proof}
It follows from the definition of $F_\stt$ given in Defintion \ref{basis:1} that $F_\stt = KF_{\stt{\downarrow}_{k-1}}$ for some $K\in P_k^\mathbb{F}(z)$. Now the result follows from the fact that the $F_\sts$'s for $\sts\in \Std_{k-1}$ are orthogonal idempotents.
\end{proof}

\begin{lemma}\label{restrictionlem2}
Let $(\lambda,k)\in \mathcal{Y}_k$ with $\varphi_n(\lambda,k)$ in an alcove. 
If $\stt,\sts\in \Std_k(\lambda)$ satisfy \newline $\stt(k-1)\neq \sts(k-1)$ then $\stt \not\approx_n \sts$.
\end{lemma}

\begin{proof}
This follows directly from Propositions \ref{edgealcove} and \ref{intervals}.
\end{proof}

\begin{lemma}\label{restrictionlem3}
Let $(\lambda,k)\in \mathcal{Y}_k$ with $\varphi_n(\lambda,k)$ in an alcove. Let $(\eta,k-1)\rightarrow (\lambda,k)$ be an edge in $\mathcal{Y}$ and let $\stt\in \Std_k(\lambda)$. Then we have
$$F_{[\stt]_n}F_{[(\eta,k-1)]_n} = \left\{ \begin{array}{ll} F_{[\stt]_n} & \mbox{if $\stt(k-1)=\eta$}\\ 0 & \mbox{otherwise.}\end{array}\right.$$
\end{lemma}

\begin{proof}
This follows from Lemmas \ref{restrictionlem1} and \ref{restrictionlem2}.
\end{proof}

\begin{lemma}\label{restrictionlem4}
Let $(\lambda,k)\in \mathcal{Y}_k$ with $\varphi_n(\lambda,k)$ in an alcove and let $\stt\in \Std_k(\lambda)$. Then we have
$$m_{\stt}F_{[\stt]_n} = m_{\stt} F_{[\stt{\downarrow}_{k-1}]_n}.$$
\end{lemma}

\begin{proof}
Inverting the equations given in Proposition \ref{s-n-d}(1) we have
\begin{eqnarray*}
m_{\stt}F_{[\stt{\downarrow}_{k-1}]_n} &=& (f_{\stt} + \sum_{\stw\in \Std_k(\lambda)}a_{\stw} f_{\stw} )  F_{[\stt{\downarrow}_{k-1}]_n} \\
&=&  (m_\stt F_{\stt} + \sum_{\stw\in \Std_k(\lambda)}a_{\stw} m_{\stw}F_{\stw} )  F_{[\stt{\downarrow}_{k-1}]_n}\\
&=&  (m_\stt F_{\stt} + \sum_{\stw\in \Std_k(\lambda)}a_{\stw} m_{\stw}F_{\stw} )  \sum_{\substack{\stv\in \Std_{k-1} \\ \stv \approx_n \stt {\downarrow}_{k-1}}} F_{\stv}.
\end{eqnarray*}
Note that, using Lemma \ref{restrictionlem1}, for any $\stw\in \Std_k(\lambda)$ and $\stv \in \Std_{k-1}$ we have $F_\stw F_\stv = 0$ unless $\stv = \stw{\downarrow}_{k-1}$ in which case we have $F_\stw F_\stv = F_\stw F_\stw = F_\stw$. So we have
$$m_\stt F_{[\stt{\downarrow}_{k-1}]_n} = m_\stt (\sum_{\substack{\sts \in \Std_k(\lambda) \\ \sts{\downarrow}_{k-1} \approx_n \stt{\downarrow}_{k-1}}} F_\sts ).$$
Note also that as $\varphi_n(\lambda,k)$ is in an alcove we have that for any $\sts, \stt\in \Std_k(\lambda)$ we have $\sts \approx_n\stt$ if and only if $\sts{\downarrow}_{k-1} \approx_n \stt{\downarrow}_{k-1}$. Thus we get
\begin{eqnarray*}
m_{\stt} F_{[\stt{\downarrow}_{k-1}]_n} &=& m_{\stt} (\sum_{\substack{ \sts\in \Std_k(\lambda) \\ \sts \approx_n \stt}} F_\sts)\\
&=& m_\stt (\sum_{\substack{\sts\in\Std_k \\ \sts \approx_n \stt}} F_\sts) \qquad \mbox{as $m_\stt F_\sts = 0$ if $\sts \notin \Std_k(\lambda)$}\\
&=& m_\stt F_{[\stt]_n}.
\end{eqnarray*}
\end{proof}

\begin{proof}[Proof of Theorem \ref{restrictioncells}]
We have
$\Delta_{k,\mathbb{Q}}^n(\lambda) = \mathbb{Q}{\rm -span} \{g_\stt \mid \stt \in \Std_k(\lambda)\}$
where $g_\stt = m_\stt F_{[\stt]_n}\otimes 1$. Now by definition we have
$G_{(\eta,k-1)} = F_{[(\eta,k-1)]_n} \otimes 1$. 
So using Lemma \ref{restrictionlem3} we obtain
$$\Delta_{k,\mathbb{Q}}^n(\lambda) G_{(\eta,k-1)} = \mathbb{Q}{\rm -span}\{ g_\stt \, : \, \stt\in \Std_k(\lambda) \, \mbox{with} \,\, \stt(k-1) = \eta\}$$
and 
$${\rm res}_{k-1}^k \Delta_{k,\mathbb{Q}}^n(\lambda) = \bigoplus_{(\eta,k-1)\rightarrow (\lambda,k)} \Delta_{k,\mathbb{Q}}^n(\lambda) G_{(\eta,k-1)}.$$
Now, using Proposition \ref{filtration}, for each edge $(\eta, k-1)\rightarrow (\lambda,k)$ in $\mathcal{Y}$ we have a chain of $P_{k-1}^{\mathcal{O}_n}(z)$-submodules
$$N_{\mathcal{O}_n}^{\rhd (\eta,k-1)} \subseteq N_{\mathcal{O}_n}^{\unrhd (\eta,k-1)} \subseteq \Delta_{k,\mathcal{O}_n}^z(\lambda)$$
where
$$ N_{\mathcal{O}_n}^{\unrhd (\eta,k-1)} = \mathcal{O}_n{\rm -span}\{m_\stt \, : \, \stt(k-1) \unrhd \eta\} =   \mathcal{O}_n{\rm -span}\{\tilde{f}_{\stt,n} \, : \, \stt(k-1) \unrhd \eta\},$$
$$ N_{\mathcal{O}_n}^{\rhd (\eta,k-1)} = \mathcal{O}_n{\rm -span}\{m_\stt \, : \, \stt(k-1) \rhd \eta\} =   \mathcal{O}_n{\rm -span}\{\tilde{f}_{\stt,n} \, : \, \stt(k-1) \rhd \eta\}.$$
Moreover we have an isomorphism
$$ N_{\mathcal{O}_n}^{\unrhd (\eta,k-1)} / N_{\mathcal{O}_n}^{\rhd (\eta,k-1)} \rightarrow \Delta_{k-1,\mathcal{O}_n}^z(\eta) \, : \, m_\stt \mapsto m_{\stt{\downarrow}_{k-1}}.$$
Now using Lemma \ref{restrictionlem4} we have $\tilde{f}_{\stt ,n} = m_\stt F_{[\stt]_n} = m_\stt F_{[\stt{\downarrow}_{k-1}]_n}$, and under the above isomorphism we have 
$$\tilde{f}_{\stt , n} \mapsto m_{\stt{\downarrow}_{k-1}} F_{[\stt{\downarrow}_{k-1}]_n} = \tilde{f}_{\stt{\downarrow}_{k-1}}.$$
Specialising to $z=n$ we get the isomorphism
$$\Delta_{k,\mathbb{Q}}^n(\lambda) G_{(\eta, k-1)} \rightarrow \Delta_{k-1, \mathbb{Q}}^n(\eta) \, : \, g_{\stt} \mapsto g_{\stt{\downarrow}_{k-1}}.$$
\end{proof}

We now describe the restriction of simple modules in an alcove. This will be used in the next section to construct bases for the radical of cell modules (and hence for the simple modules).

\begin{theorem}\label{restrictionsimple}
Let  $(\lambda, k)\in \mathcal{Y}_k$. If    $\varphi_n\pointla $ is in  the $j$-th alcove then
\begin{equation}\label{ressimple}
\res^k_{k-1}L_{k,\QQ}^n(\lambda) \cong \bigoplus_{(\mu,k-1)\rightarrow (\lambda,k)} L_{k-1,\QQ}^n(\mu)
\end{equation}
where the sum is taken over all edges $(\mu, k-1) \rightarrow \pointla $  in $\mathcal{Y}$ such that  $\varphi_n(\mu, k-1)$ is either in the $j$-th alcove or on the $(j-1)$-th wall.
\end{theorem}

Note that if $k$ is even, $n=0$ and $\lambda = \varnothing$ then both sides of the equation in (\ref{ressimple}) are zero and so the result trivially holds in this case.

In order to prove this theorem we will use the following lemma.

\begin{lemma}\label{maximpliesnotwall}
Let $(\lambda,k)\in \mathcal{Y}_k$ with $\varphi_n(\lambda,k)$ in the $j$-th alcove. Suppose that  there exists an edge  $(\eta , k-1)\rightarrow (\lambda,k)$ in $\mathcal{Y}$ with $\varphi_n(\eta,k-1)$ on the $j$-th wall. Then either
\begin{enumerate}
\item $k$ is even, $\lambda = \eta$ and we have $(\mu,k):=(\eta + \varepsilon_j,k)\in \mathcal{Y}_k$, or
\item $k$ is odd, $\lambda = \eta - \varepsilon_j$ and we have $(\mu,k) := (\eta,k)\in \mathcal{Y}_k$.
\end{enumerate}
In both cases there is an edge
 $(\eta, k-1)\rightarrow (\mu,k)$ in $\mathcal{Y}$ and $\varphi_n(\mu,k) = s_{0,j}(\varphi_n(\lambda ,k))$.
\end{lemma}
 
\begin{proof}
Suppose that $(\eta,k-1)\rightarrow (\lambda,k)$ is an edge in $\mathcal{Y}$ and $\varphi_n(\eta,k-1)$ is on the $j$-th wall. Then using Proposition \ref{edgealcove} we have either $k$ is even and $\varphi_n(\eta,k-1) = \varphi_n(\lambda,k) - \varepsilon_0$ or $k$ is odd and $\varphi_n(\eta,k-1) = \varphi_n(\lambda,k)+\varepsilon_j$.  In the first case we define $(\mu,k)=(\eta + \varepsilon_j, k)\in \mathcal{Y}_k$ (note that $\mu$ is a partition as $\varphi_n(\lambda,k)$ is in the $j$-th alcove). In the second case we define $(\mu,k)=(\eta,k)\in \mathcal{Y}_k$. In both cases we have  an edge
 $(\eta, k-1)\rightarrow (\mu,k)$ and $\varphi_n(\mu,k) = s_{0,j}(\varphi_n(\lambda,k))$ as required.  
\end{proof}

\begin{proof}[Proof of Theorem \ref{restrictionsimple}]
We prove this result by downward induction on the degree of partitions in each block.
First consider the case where $(\lambda, k)$ has maximal degree in its block. Using Theorems \ref{cfcells} and \ref{restrictioncells} we have
$${\rm res}_{k-1}^k L_{k,\mathbb{Q}}^n(\lambda) = {\rm res}_{k-1}^k \Delta_{k,\mathbb{Q}}^n(\lambda) \cong \bigoplus_{(\eta,k-1)\rightarrow (\lambda,k)} \Delta_{k-1, \mathbb{Q}}^n(\eta)$$
where the sum is over all edges $(\eta,k-1)\rightarrow (\lambda,k)$ in $\mathcal{Y}$. Now using Proposition \ref{edgealcove} and Lemma \ref{maximpliesnotwall}, for any such edge we have that either $\varphi_n(\eta,k-1)$ is in the $j$-th alcove or on the $(j-1)$-th wall. Note also that as $L_{k,\mathbb{Q}}^n(\lambda)$ is self dual, so is its restriction and so we have $\Delta_{k-1,\mathbb{Q}}^n(\eta) = L_{k-1,\mathbb{Q}}^n(\eta)$ for any edge $(\eta, k-1)\rightarrow (\lambda, k)$. So we are done in this case.

Now assume that $(\lambda,k)$ is not maximal in its block. Then using Theorem \ref{cfcells} we have $(\mu,k)\in \mathcal{Y}_k$ with $\varphi_n(\mu,k) = s_{0,j}(\varphi_n(\lambda,k))$ in the $(j+1)$-th alcove, $|\mu|>|\lambda|$ and an exact sequence
\begin{equation*}
0 \rightarrow {\rm res}_{k-1}^k L_{k,\mathbb{Q}}^n(\mu) \rightarrow {\rm res}_{k-1}^k \Delta_{k,\mathbb{Q}}^n(\lambda) \rightarrow {\rm res}_{k-1}^k L_{k,\mathbb{Q}}^n(\lambda) \rightarrow 0.
\end{equation*}
Using Theorem \ref{restrictioncells} and induction, this short exact sequence becomes
\begin{equation}\label{sesrestriction}
0 \rightarrow \bigoplus_{(\theta,k-1)\rightarrow (\mu,k)} L_{k-1,\mathbb{Q}}^n(\theta) \rightarrow \bigoplus_{(\eta,k-1) \rightarrow (\lambda,k)} \Delta_{k-1,\mathbb{Q}}^n(\eta) \rightarrow {\rm res}_{k-1}^k L_{k,\mathbb{Q}}^n(\lambda) \rightarrow 0.
\end{equation}
Note that for all $(\theta, k-1)$ appearing in (\ref{sesrestriction}) we have that $\varphi_n(\theta, k-1)$ is either in the $(j+1)$-th alcove or on the $j$-th wall, and  for all $(\eta,k-1)$ appearing in (\ref{sesrestriction}) we have that $\varphi_n(\eta,k-1)$ is either in the $j$-th alcove, on the $j$-th wall or on the $(j-1)$-th wall.

As $L_{k,\mathbb{Q}}^n(\lambda)$ is self-dual, so is its restriction and so we have a surjection
$$\bigoplus_{(\eta,k-1) \rightarrow (\lambda,k)} L_{k-1,\mathbb{Q}}^n(\eta) \rightarrow {\rm res}_{k-1}^k L_{k,\mathbb{Q}}^n(\lambda)$$
where the $(\eta,k-1)$ appearing in the direct sum are as in (\ref{sesrestriction}).
Now, any $L_{k-1,\mathbb{Q}}^n(\eta)$ with $\varphi_n(\eta, k-1)$ in the $j$-th alcove, or on the $(j-1)$-th wall does not appear on the left hand side of the short exact sequence given in (\ref{sesrestriction}) and so must appear in ${\rm res}_{k-1}^k L_{k,\mathbb{Q}}^n(\lambda)$. Finally, for any $L_{k,\mathbb{Q}}^n(\eta)$ with $\varphi_n(\eta,k-1)$ on the $j$-th wall we have an edge $(\eta,k-1)\rightarrow (\mu,k)$ in $\mathcal{Y}$ by Lemma \ref{maximpliesnotwall}, and thus $L_{k-1,\mathbb{Q}}^n(\eta)$ appears on the left hand side of (\ref{sesrestriction}). As the composition factors of $\bigoplus_{(\eta,k-1) \rightarrow (\lambda,k)} \Delta_{k-1,\mathbb{Q}}^n(\eta)$ are multiplicity free, such simple modules cannot appear in ${\rm res}_{k-1}^k L_{k,\mathbb{Q}}^n(\lambda)$. This proves the result.  
\end{proof}

\section{A basis for the radical of a cell module}\label{the radical}

Throughout this section we continue to fix $n\in \mathbb{Z}_{\geq 0}$. The aim of this section is to  provide an explicit basis for the radical of the bilinear form on any cell module
$\standard{k,\QQ}n(\lambda)$ for $\algebra{k}\QQ(n)$, and hence also for any simple module.

\medskip

Using Proposition \ref{edgealcove} we can make the following definition.

\begin{definition}\label{npermissiblepathas}
Let $(\lambda, k)\in \mathcal{Y}$. 
 Given $\stt \in \Std_k(\lambda)$, we say that $\stt$ is $n$-permissible  if the following conditions hold.
\begin{enumerate}[label=(\arabic{*}), ref=\arabic{*},leftmargin=0pt,itemindent=1.5em]
\item  if $\varphi_n\pointla $   in the first alcove, 
 then $\varphi_n(\stt(i)) $ belongs to the first alcove for all $0\leq i \leq k$. 
\item if $\varphi_n(\lambda,k)$  in the $j$-th alcove for some $j\geq 2$,  then $\stt$ is a path which last enters the $j$-th alcove from the $(j-1)$-th wall. 
\end{enumerate}
We denote by $\Std_k^n(\lambda)$ the set of all $n$-permissible $\stt\in \Std_k(\lambda)$.
\end{definition}  

Note that if $\varphi_n(\lambda,k)$ is on a wall then any $\stt\in \Std_k(\lambda)$ is $n$-permissible and so $\Std_k^n(\lambda)=\Std_k(\lambda)$.

\begin{example}
Let $n=2$ and $k=6$ as in \hyperref[piccy]{Figure~\ref*{piccy}}. 
We have paths 
$$
\sts=(\varnothing,\varnothing, (1),\varnothing,(1), \varnothing,\varnothing)
\quad  \stt=(\varnothing,\varnothing, (1),(1),(1), \varnothing,\varnothing)
\quad
\stu=(\varnothing,\varnothing, (1),(1),(2), (2), (3))
$$
The paths $\sts$ and $\stu$ are $2$-permissible but the  path $\stt$ is not. \end{example}

\begin{lemma}\label{dimsimples}
Let $(\lambda, k)\in \mathcal{Y}_k$.  We have that 
   $
  \dim_\QQ(\simple{k,\QQ}n(\lambda)) = 
  |\Std_k^n(\lambda)|.  
  $
 \end{lemma}

\begin{proof}
This follows  from \hyperref[restrictionsimple]{Theorem~\ref*{restrictionsimple}} by induction on $k$. Note that for $k$ even, $n=0$ and $\lambda = \varnothing$  we have that $\Std_k^n(\lambda)=\emptyset$ and $L_{k,\mathbb{Q}}^n(\lambda) = 0$  so the result also holds in this case.
\end{proof}

\begin{example}
Let $n=2$ and $k=6$ as in \hyperref[piccy]{Figure~\ref*{piccy}}.  We have that 
$$\dim_\QQ(L_{6,\QQ}^{2}(\varnothing) )= 4 \quad
\quad  \dim_\QQ(L_{6,\QQ}^{2}((1))) = 4.$$  
\end{example}

The next result lifts the dimension result given in   \hyperref[dimsimples]{Lemma~\ref*{dimsimples}} to an explicit basis.

\begin{theorem}\label{radicalofform}
Let $(\lambda,k)\in \mathcal{Y}_k$.
Then the set   $\{g_{\stt } \mid \stt \in \Std_k(\lambda) \setminus \Std_k^n(\lambda)\}$  forms a $\QQ$-basis for  $\rad\standard{k,\QQ}n(\lambda)$.
\end{theorem}

\begin{proof}
If $k$ is even, $n=0$ and $\lambda = \varnothing$ then $\rad\Delta_{k,\mathbb{Q}}^0(\varnothing)=\Delta_{k,\mathbb{Q}}^0(\varnothing)$ and $\Std_k^0(\varnothing)=\emptyset$ so we're also done.

Now, if $\varphi_n(\lambda,k)$ is on a wall then by Remark \ref{remarkwall} we have that  $\Delta_{k,\mathbb{Q}}^n(\lambda)$ is simple and so $\rad \Delta_{k,\mathbb{Q}}^n(\lambda)=0$. As we have $\Std_k^n(\lambda)=\Std_k(\lambda)$ in this case, we are also done.

Now suppose that $L_{k,\mathbb{Q}}^n(\lambda)\neq 0$ and $\varphi_n(\lambda,k)$ is in an alcove, say the $j$-th alcove. We proceed by induction on $k$. If $k=1$ then $\lambda=\varnothing$. We have that $\Delta_{1,\mathbb{Q}}^n(\varnothing)$ is one-dimensional, hence simple and so $\rad\Delta_{1,\mathbb{Q}}^n(\varnothing) = 0$. As we have $\Std_1(\varnothing)=\Std_1^n(\varnothing)$ the result holds in this case.

 Assume that the result holds for $k-1$ and prove it for $k$. Using Theorem \ref{restrictioncells} we have an isomorphism
\begin{equation}\label{restrictioniso}
{\rm res}_{k-1}^k \Delta_{k,\mathbb{Q}}^n(\lambda) \rightarrow \bigoplus_{(\eta , k-1) \rightarrow (\lambda,k)}\Delta_{k-1, \mathbb{Q}}^n(\eta) \, : \, g_{\stt} \mapsto g_{\stt{\downarrow}_{k-1}}
\end{equation}
for any $\stt\in \Std_k(\lambda)$, where the direct sum is over all edges $(\eta,k-1)\rightarrow (\lambda,k)$ in $\mathcal{Y}$. So the pre-image of $\bigoplus_{(\eta,k-1)\rightarrow (\lambda,k)} \rad\Delta_{k-1,\mathbb{Q}}^n(\eta)$ under the isomorphism given in (\ref{restrictioniso}) gives precisely the (unique) smallest submodule of ${\rm res}_{k-1}^k \Delta_{k,\mathbb{Q}}^n(\lambda)$ with semisimple quotient. Now, note that
$${\rm res}_{k-1}^k \Delta_{k,\mathbb{Q}}^n(\lambda) / {\rm res}_{k-1}^k \rad \Delta_{k,\mathbb{Q}}^n(\lambda) = {\rm res}_{k-1}^k  \left( \Delta_{k,\mathbb{Q}}^n(\lambda) / \rad \Delta_{k,\mathbb{Q}}^n(\lambda) \right) = {\rm res}_{k-1}^k L_{k,\mathbb{Q}}^n(\lambda)$$
which is semisimple using Theorem \ref{restrictionsimple}. Thus, under the isomorphism given in (\ref{restrictioniso}), the pre-image of $\bigoplus_{(\eta,k-1)\rightarrow (\lambda,k)} \rad\Delta_{k-1,\mathbb{Q}}^n(\eta)$ must be contained in ${\rm res}_{k-1}^k \rad \Delta_{k,\mathbb{Q}}^n(\lambda)$. 

Using Proposition \ref{edgealcove}, we know that for any edge $(\eta,k-1)\rightarrow (\lambda,k)$ in $\mathcal{Y}$ we have that either $\varphi_n(\eta,k-1)$ is in the $j$-th alcove, on the $(j-1)$-th wall or on the $j$-th wall.
If $\varphi_n(\eta,k-1)$ is in the $j$-th alcove, then by induction and using the isomorphism given in (\ref{restrictioniso}) we can deduce that any $g_{\stt}$ with $\stt\in \Std_k(\lambda)\setminus \Std_k^n(\lambda)$ and $\stt(k-1) = (\eta,k-1)$ must be in $\rad\Delta_{k,\mathbb{Q}}^n(\lambda)$.

We claim that any $g_{\stt }$ with $\stt {\downarrow}_{k-1}\in \Std_{k-1}(\eta)$ and $\varphi_n(\eta,k-1)$ on the $j$-th wall also belongs to $\rad\Delta_k^n(\lambda)$. Then, using \hyperref[dimsimples]{Lemma~\ref*{dimsimples}} the result will follow by a dimension count.

 Now, if $\varphi_n(\eta,k-1)$ is on the $j$-th wall, then $\standard{k-1,\QQ}n(\eta)$ is simple, so it is enough to show that one $g_{\stt }$ for $\stt {\downarrow}_{k-1}\in \Std_{k-1}{(\eta)}$ is in the radical of $\Delta_{k,\QQ}^n(\lambda)$.
 Suppose  $|\mu| = m$, then we choose $\stt $ to be any path satisfying $\stt (0)=\stt (1) = \ldots = \stt (k-2m)=\varnothing $ (the rest of the path can be taken by adding boxes along the rows of $\eta$ at every even step and with the last step going from $(\eta,k-1)$ to $(\lambda,k)$). It is easy to see that for such a path $\stt$ we have $[\stt ]_n^{(\lambda,k)} = \{\stt\}$ and $[\stt{\downarrow}_{k-1}]_n^{(\eta,k-1)} = \{\stt {\downarrow}_{k-1}\}$. Thus we have
$$\tilde{f}_{\stt,n} = f_{\stt} \in P_k^{\mathcal{O}_n}(z) \quad \mbox{and} \quad g_\stt = f_{\stt} \otimes 1 \in P_k^\mathbb{Q}(n),$$
and similarly $\tilde{f}_{\stt {\downarrow}_{k-1},n} = f_{\stt{\downarrow}_{k-1}}\in P_{k-1}^{\mathcal{O}_n}(z)$ and $g_{\stt{\downarrow}_{k-1}} = f_{\stt{\downarrow}_{k-1}} \otimes 1 \in P_{k-1}^\mathbb{Q}(n)$.
Now by Proposition \ref{innerproductg} we have that $\langle g_{\stt}, g_{\stu}\rangle = 0$ for all $\stu\in \Std_k(\lambda)$ with $\stu \neq \stt$ (as $\stt$ is alone in its residue class). It remains to show that $\langle g_\stt, g_\stt \rangle =0$. Now we have
$$\langle g_\stt, g_\stt \rangle = \langle f_\stt, f_\stt \rangle \otimes 1$$
Let $\sts =\stt{\downarrow}_{k-1}\in \Std_{k-1}(\eta)$. Then we have
\begin{equation}\label{recinnerproduct}
\langle f_\stt,f_\stt \rangle = \gamma_{(\eta,k-1)\rightarrow (\lambda,k)}\langle f_\sts,f_\sts \rangle
\end{equation}
where the coefficient $\gamma_{(\eta,k-1)\rightarrow (\lambda,k)}$ is given in  Proposition \ref{branching}.
Now as $\varphi_n(\lambda,k)$ is in the $j$-th alcove and $\varphi_n(\eta,k-1)$ is on the $j$-th wall we know, from Lemma \ref{maximpliesnotwall} and Proposition \ref{branching} that either
\begin{enumerate}
\item[(I)] $k$ is even, $\lambda = \eta$, $\eta + \varepsilon_j$ is a partition and 
\begin{equation}\label{gammaeven}
\gamma_{(\eta, k-1)\rightarrow (\lambda , k)} = \frac{\prod_{\beta\in A(\lambda)}(z-c(\beta) -|\lambda|)}{\prod_{\beta\in R(\lambda)}(z-c(\beta)-|\lambda|)}\,  r'
\end{equation}
for some $r'\in \mathbb{Q}$, or
\item[(II)] $k$ is odd and $\lambda = \eta - \varepsilon_j$ and 
$$\gamma_{(\eta, k-1)\rightarrow (\lambda , k)} = \frac{(z-c(\alpha)-|\lambda| -1)}{(z-c(\alpha) -|\lambda|)}\,  r$$
where $\alpha$ denotes the box removed from row $j$ of $\eta$ to get $\lambda$ and  $r\in \mathbb{Q}$.
\end{enumerate}
Now as $\varphi_n(\eta,k-1)$ is on the $j$-th wall we have
$$\eta_j - j = \left\{ \begin{array}{ll} n-1-|\eta| & \mbox{if $k$ is even}\\ n-|\eta| & \mbox{if $k$ is odd}\end{array}\right.$$
In case (I) the partition $\lambda = \eta$ has an addable box $\beta$ in row $j$ with content 
$$c(\beta) = \eta_j +1 -j = n-|\lambda|.$$
Thus we get from (\ref{recinnerproduct}) that $\langle g_\stt ,g_\stt \rangle = 0$. (Note that there is no possible cancellation in (\ref{gammaeven}) as the content of removable boxes and addable boxes of a given partition are all distinct.) In case (II) the content of $\alpha$ is given by
$$c(\alpha) = \eta_j - j = n-|\eta| = n-|\lambda| -1.$$
Thus we get from (\ref{recinnerproduct}) that $\langle g_\stt , g_\stt \rangle =0$ in this case as well.
\end{proof}

\begin{corollary}\label{MAIN} Let $(\lambda,k)\in \mathcal{Y}_k$. Then the set $\{ g_\stt + \rad \Delta_{k,\mathbb{Q}}^n(\lambda) \mid \stt \in \Std_k^n(\lambda)\}$ form a basis $L_{k,\mathbb{Q}}^n(\lambda)$. 
Moreover, if $\varphi_n(\lambda,k)$ is in the first alcove, then for all $\stt\in \Std_k^n(\lambda)$ we have $\tilde{f}_{\stt , n} = f_\stt$ and $g_\stt = f_\stt \otimes 1$. 
\end{corollary}

\begin{proof}
The first part follows directly from Theorem \ref{radicalofform}. The second part follows from the fact that when $\varphi_n(\lambda,k)$ is in the first alcove we have $\stt(i)$ in the first alcove for any $\stt\in \Std_k^n(\lambda)$ and $0\leq i\leq k$. This implies that any such $\stt$ is alone in its residue class. Hence $\tilde{f}_{\stt ,n} = f$ in this case. 
\end{proof}

 \section{Monotone Convergence of  Kronecker coefficients}

In this final section, we apply \hyperref[MAIN]{Theorem~\ref*{radicalofform}} to the study of the Kronecker coefficients.  
These coefficients appear in the classical representation theory of the symmetric group. Denote by $\mathfrak{S}_n$ the symmetric group of degree $n$. The simple $\mathbb{Q}\mathfrak{S}_n$-modules, known as the Specht modules, are indexed by partitions of $n$. We will use a slightly unusual notation for these partitions; the reason for this will become clear in what follows.

For a partition $\lambda = (\lambda_1, \lambda_2, \lambda_3, \ldots )$ and $n\in \mathbb{Z}_{\geq 0}$ we define 
$$\lambda_{[n]} = (n-|\lambda|, \lambda_1, \lambda_2, \lambda_3, \ldots ).$$
Note that for $n$ sufficiently large $\lambda_{[n]}$ is a partition of $n$. Moreover, any partition of $n$ can be written as $\lambda_{[n]}$ for some partition $\lambda$. 

For each partition $\lambda_{[n]}$ we denote by $\mathbf{S}(\lambda_{[n]})$ the corresponding Specht module for $\mathbb{Q}\mathfrak{S}_n$. Now for $\lambda_{[n]}, \mu_{[n]}, \nu_{[n]}$ partitions of $n$, the Kronecker coefficient $g_{\lambda_{[n]},\mu_{[n]}}^{\nu_{[n]}}$ is defined by
$$g_{\lambda_{[n]},\mu_{[n]}}^{\nu_{[n]}} = \dim_{\mathbb{Q}} {\rm Hom}_{\mathbb{Q}\mathfrak{S}_n} (\mathbf{S}(\lambda_{[n]})\otimes \mathbf{S}(\mu_{[n]}), \mathbf{S}(\nu_{[n]})).$$

Murnaghan discovered an amazing limiting phenomenon satisfied by the Kronecker coeficients; as we increase the length of the first row of the indexing partitions the sequence of Kronecker coefficients stabilises  (see \cite{murn,MR1243152,MR1725703} for various proofs). 
This is illustrated in the following example.

 \begin{example}
We have the following decomposition of tensor products of Specht modules:
$$\begin{array}{cll}   
 n=2 	&\quad &    \mathbf{S}(1^2) \otimes \mathbf{S}(1^2)  = \mathbf{S}(2)  \\
 n=3  &\quad &\mathbf{S}(2,1) \otimes \mathbf{S}(2,1)  = \mathbf{S}(3) \oplus \mathbf{S}(2,1) \oplus \mathbf{S}(1^3) \\
n=4  &\quad &\mathbf{S}(3,1) \otimes \mathbf{S}(3,1)  = \mathbf{S}(4) \oplus \mathbf{S}(3,1)\oplus \mathbf{S}(2,1^2) \oplus \mathbf{S}(2^2)  
\end{array}
$$ 
 at which point the product stabilises, i.e. for all $n\geq4$, we have 
 $$\mathbf{S}(n-1,1) \otimes \mathbf{S}(n-1,1) = \mathbf{S}(n) \oplus \mathbf{S}(n-1,1)\oplus \mathbf{S}(n-2,1^2) \oplus \mathbf{S}(n-2,2).$$
   \end{example}

The limit of the sequence $g_{\lambda_{[n]},\mu_{[n]}}^{\nu_{[n]}}$ as $n$ increases are known as the {\sf stable (or reduced) Kronecker coefficients} and denoted by $\bar{g}_{\lambda,\mu}^\nu$. So for $N$ sufficiently large we have
$$g_{\lambda_{[N+n]},\mu_{[N+n]}}^{\nu_{[N+n]}} = \bar{g}_{\lambda,\mu}^\nu \qquad \mbox{for all $n\geq 1$}.$$

This stability is rather startling from the point of view of the symmetric group.  
However,  in \cite{MR3314819}, the Kronecker  coefficients were given a new interpretation in the setting of the partition algebra  where this phenomenon becomes very natural. Using the Schur-Weyl duality between the symmetric group and the partition algebra we obtain a new interpretation of the Kronecker coefficients as follows. Let $\lambda_{[n]}, \mu_{[n]}, \nu_{[n]}$ be partitions of $n$ with $|\lambda|=r$ and $|\mu|=s$ and write $p=r+s$. We write $P_{2r,2s}^\mathbb{Q}(n) = P_{2r}^\mathbb{Q}(n)\otimes P_{2s}^\mathbb{Q}(n)\subseteq P_{2p}^\mathbb{Q}(n)$ and write ${\rm res}^{2p}_{2r,2s}$ for the restriction functor from   $P_{2p}^\mathbb{Q}(n)$-modules to  $P_{2r}^\mathbb{Q}(n)\otimes P_{2s}^\mathbb{Q}(n)$-modules. Then we have
 \begin{equation}\label{gw}
 g^{\nu_{[n]}}_{\lambda_{[n]},\mu_{[n]}} = \left\{ \begin{array}{ll}
\dim_{\mathbb{Q}} {\rm Hom}_{P_{2r,2s}^\mathbb{Q}(n)} (L_{2r,\mathbb{Q}}^n(\lambda) \boxtimes L_{2s,\mathbb{Q}}^n(\mu), {\rm res}^{2p}_{2r,2s} L_{2p,\mathbb{Q}}^n(\nu))
 & \mbox{if $|\nu|\leq p$}\\  0 & \mbox{otherwise} \end{array} \right.
 \end{equation}
(see   \cite[Section 3]{MR3314819}).
 
Note that as $|\lambda|=r$ and $|\mu|=s$ we have $L_{2r,\mathbb{Q}}^n(\lambda) = \Delta_{2r,\mathbb{Q}}^n(\lambda)$ and $L_{2s,\mathbb{Q}}^n(\mu) = \Delta_{2s,\mathbb{Q}}^n(\mu)$.
Now for sufficiently large values of $n$ the partition algebra $P_{2p}^\mathbb{Q}(n)$ is semisimple and $L_{2p,\mathbb{Q}}^n(\nu) = \Delta_{2p,\mathbb{Q}}^n(\nu)$  and so we have a new interpretation of the stable Kronecker coefficients as
$$\bar{g}_{\lambda,\mu}^\nu = \dim_{\mathbb{Q}} {\rm Hom}_{P_{2r,2s}^\mathbb{Q}(n)} (\Delta_{2r,\mathbb{Q}}^n(\lambda) \boxtimes \Delta_{2s,\mathbb{Q}}^n(\mu), {\rm res}^{2p}_{2r,2s} \Delta_{2p,\mathbb{Q}}^n(\nu))$$
for all $n$ sufficiently large (see  \cite[Corollary 3.2]{MR3314819}).

Brion proved in \cite[Section 3.4, Corollary 1]{MR1243152} that the sequence of Kronecker coefficients $g_{\lambda_{[n]}, \mu_{[n]}}^{\nu_{[n]}}$ not only stabilises but is also monotonic. More precisely, he showed that
 $$
 g^{\nu_{[n+1]}}_{\lambda_{[n+1]}, \mu_{[n+1]}}
 \geq 
 g^{\nu_{[n ]}}_{\lambda_{[n] }, \mu_{[n ]}} 
 $$
\noindent Briant asked whether this monotonicity could also be explained in the context of the partition algebra. In the rest of this section we will show that it does.

 Using (\ref{gw}) we need to study simple modules for the partition algebra. We first make the observation that we only need to consider simple modules labelled by partitions in the first alcove. More precisely we have the following lemma.

\begin{lemma}\label{kcimpliesfirstalcove}
Let $n,k\in \mathbb{Z}_{\geq 0}$ and let $\lambda$ be a partition with $|\lambda|\leq k$. Then $\lambda_{[n]}$ is a partition if and only if $\varphi_n(\lambda,2k)$ is in the first alcove.
\end{lemma}

\begin{proof}
We have that $\lambda_{[n]}$ is a partition if and only if $n-|\lambda| \geq \lambda_1$. But this holds precisely when $n-|\lambda| > \lambda_1 -1$ which is exactly the condition for $\varphi_n(\lambda,2k)$ to be in the first alcove.
\end{proof}

\begin{lemma}\label{finallemmaone}
Let $(\nu,k)\in \mathcal{Y}_k$ with $\varphi_n(\nu,k)$ in the first alcove. Then $\varphi_{n+1}(\nu,k)$ is also in the first alcove. Moreover, if $\stt \in \Std_k^n(\nu)$ then $\stt \in \Std_k^{n+1}(\nu)$.
\end{lemma}

\begin{proof}
We have $\varphi_n(\nu,k)$ in the first alcove if and only if 
$$\left\{ \begin{array}{ll} n-|\nu| > \nu_1-1 & \mbox{if $k$ is even}\\ n-1-|\nu| > \nu_1 -1 & \mbox{if $k$ is odd}\end{array} \right.$$
and $\varphi_{n+1}(\nu,k)$ in the first alcove when the same condition holds with $n$ replaced by $n+1$. So clearly we have that if $\varphi_n(\nu,k)$ is in the first alcove then so is $\varphi_{n+1}(\nu,k)$. Now $\stt\in \Std_k^n(\nu)$ precisely when every $\varphi_n(\stt(i))$ for $0\leq i \leq k$ belongs to the first alcove. Hence if $\stt \in \Std_k^n(\nu)$ then $\stt \in \Std_k^{n+1}(\nu)$ as required.
\end{proof}

\begin{lemma}\label{finallemmatwo}
Let $(\lambda,k)\in \mathcal{Y}_{k}$ with $\varphi_n(\lambda,k)$ in the first alcove and $|\lambda| = \lfloor k/2 \rfloor$. Then we have
$$F_{[(\lambda,k)]_n} = F_{[(\lambda,k)]_{n+1}} = F_{(\lambda,k)}.$$
Moreover, $L_{k,\mathbb{Q}}^n(\lambda)$ (respectively $L_{k,\mathbb{Q}}^{n+1}(\lambda)$)  is alone in its block and we have
$$
\Delta_{k,\mathbb{Q}}^n(\lambda) = L_{k,\mathbb{Q}}^n(\lambda) \quad \mbox{and} \quad 
\Delta_{k,\mathbb{Q}}^{n+1}(\lambda) = L_{k,\mathbb{Q}}^{n+1}(\lambda). $$
\end{lemma}

\begin{proof}
Let $\stt\in \Std_{k}(\lambda)$. As $|\lambda| = \lfloor k/2 \rfloor$ we have that every step in $\stt$ is of the form
$\varphi_n(\stt(i))=\varphi_n(\stt(i-1)) +\varepsilon_j$ for some $j\geq 1$ if $i$ is even and $\varphi_n(\stt(i)) = \varphi_n(\stt(i-1)) -\varepsilon_0$ is $i$ is odd (and similarly for $n+1$). As $\varphi_n(\lambda,k)$ is in the first alcove, so is $\varphi_{n+1}(\lambda,k)$, by Lemma \ref{finallemmaone}.  It follows from Proposition \ref{intervals} and Lemma \ref{maximpliesnotwall} that $[\stt]_n=[\stt]_{n+1} = \{ \stt\}$. Note that this holds for any $\stt\in \Std_k(\lambda)$ and so we get $F_{[(\lambda,k)]_n} = F_{[(\lambda,k)]_{n+1}} = F_{(\lambda,k)}$ by definition. 
This implies that the simple module $ L_{k,\mathbb{Q}}^n(\lambda)$ (respectively $ L_{k,\mathbb{Q}}^{n+1}(\lambda) $) is alone in its block and hence we get $\Delta_{k,\mathbb{Q}}^n(\lambda) = L_{k,\mathbb{Q}}^n(\lambda)$ (respectively $\Delta_{k,\mathbb{Q}}^{n+1}(\lambda) = L_{k,\mathbb{Q}}^{n+1}(\lambda) $) as required.
\end{proof}

\begin{lemma}\label{finallemmathree}
Let $(\nu,k)\in \mathcal{Y}_k$. Then we have
$$\mathbb{F}{\rm-span} \{ \tilde{f}_{\stt , n} \mid \stt \in \Std_k(\nu)\setminus \Std_k^n(\nu)\} = \mathbb{F}{\rm-span} \{ f_{\stt}\mid \stt \in \Std_k(\nu)\setminus \Std_k^n(\nu)\}.$$
\end{lemma}

\begin{proof}
Recall that for any $\stt \in \Std_k(\nu)$ we have
\begin{equation}\label{changeofbasis}
\tilde f_{\stt , n } = f_{\stt } +\sum_{  \begin{subarray}{c}\sts \succ \stt \\ \sts \approx_n\stt \end{subarray}} a_{\sts } f_{\sts }  \end{equation}
where the sum is over $\sts \in \Std_k(\nu)$. Now using Proposition \ref{intervals}, if $\sts \approx_n \stt$ then $\sts(i)=\stt(i)$ whenever $\varphi_n(\stt(i))$ is on a wall. It follows that $\stt$ is $n$-permissible if and only if $\sts$ is $n$-permissible. 
Now the result follows from the fact that the change of basis given in (\ref{changeofbasis}) is unitriangular.
\end{proof}
 
\begin{proposition}\label{finalprop}
Let $(\nu,k)\in \mathcal{Y}_k$ with $\varphi_n(\nu,k)$ in the first alcove. Define $\mathcal{O}_n$- and $\mathcal{O}_{n+1}$-modules
\begin{eqnarray*}
\Delta_{k,\mathcal{O}_n}^{z,n}(\nu) &=& \mathcal{O}_n{\rm-span} \{ \tilde{f}_{\stt,n}\mid \stt \in \Std_k(\nu) \setminus \Std_k^n(\nu) \},\\
\Delta_{k,\mathcal{O}_{n+1}}^{z,n+1}(\nu) &=& \mathcal{O}_{n+1}{\rm-span} \{ \tilde{f}_{\stt,n+1} \mid \stt \in \Std_k(\nu) \setminus \Std_k^{n+1}(\nu) \}
\end{eqnarray*}
and $\mathbb{F}$-vector spaces $\Delta_{k,\mathbb{F}}^{z,n}(\nu) = \Delta_{k,\mathcal{O}_n}^{z,n}(\nu) \otimes_{\mathcal{O}_n} \mathbb{F}$ and 
$\Delta_{k,\mathbb{F}}^{z,n+1}(\nu) = \Delta_{k,\mathcal{O}_{n+1}}^{z,n+1}(\nu) \otimes_{\mathcal{O}_{n+1}} \mathbb{F}$.
Then we have inclusions of $\mathbb{F}$-vector spaces 
$$\Delta_{k,\mathbb{F}}^{z,n+1}(\nu)  \subseteq \Delta_{k,\mathbb{F}}^{z,n}(\nu) \subseteq \Delta_{k,\mathbb{F}}^{z}(\nu)  .$$
\end{proposition}

\begin{proof}
This follows directly from Lemmas \ref{finallemmaone} and \ref{finallemmathree}.
\end{proof}

\begin{corollary}(see \cite[Section 3.4, Corollary 1]{MR1243152})
Let  $\lambda_{[n]},\mu_{[n]}, \nu_{[n]}$ be partitions of $n$, then 
$$g_{\lambda_{[n]}, \mu_{[n]}}^{\nu_{[n]}} \leq g_{\lambda_{[n+1]}, \mu_{[n+1]}}^{\nu_{[n+1]}}.$$
\end{corollary}

\begin{proof}
Let $r=|\lambda|$, $s=|\mu|$ and $p=r+s$. We can assume that $|\nu| \leq p$ as otherwise $g_{\lambda_{[n]},\mu_{[n]}}^{\nu_{[n]}} = g_{\lambda_{[n+1]},\mu_{[n+1]}}^{\nu_{[n+1]}} = 0$.
Using Theorem \ref{radicalofform} we have
$$\Delta_{2p,\mathcal{O}_n}^{z,n}(\nu) \otimes_{\mathcal{O}_n} \mathbb{Q} = \rad \Delta_{2p,\mathbb{Q}}^n(\nu) \quad
\mbox{and} 
\quad \Delta_{2p,\mathcal{O}_n}^{z,n+1}(\nu) \otimes_{\mathcal{O}_{n+1}} \mathbb{Q} = \rad \Delta_{2p,\mathbb{Q}}^{n+1}(\nu).$$
Now consider $F_{(\lambda,2r)} \otimes F_{(\mu,2s)} \in P_{2r}^\mathbb{F}(z) \otimes P_{2s}^\mathbb{F}(z) \subseteq P_{2p}^{\mathbb{F}} (z)$. Using Lemma \ref{finallemmatwo} we have
$$F_{(\lambda,2r)} \otimes F_{(\mu,2s)} \in P_{2p}^{\mathcal{O}_n}(z) \cap P_{2p}^{\mathcal{O}_{n+1}}(z)$$
and 
\begin{eqnarray*}
\Delta_{2p,\mathcal{O}_n}^{z,n}(\nu) (F_{(\lambda,2r)} \otimes F_{(\mu,2s)})\otimes_{\mathcal{O}_n} \mathbb{Q} &=& (\rad \Delta_{2p,\mathbb{Q}}^n(\nu))(G_{(\lambda,2r)} \otimes G_{(\mu,2s)})\\
&=& (\bar{g}_{\lambda, \mu}^{\nu}-g_{\lambda_{[n]}, \mu_{[n]}}^{\nu_{[n]}}) (\Delta_{2r,\mathbb{Q}}^n(\lambda) \otimes \Delta_{2s,\mathbb{Q}}^n (\mu)).
\end{eqnarray*}
Similarly we have
$$\Delta_{2p,\mathcal{O}_{n+1}}^{z,n+1}(\nu) (F_{(\lambda,2r)} \otimes F_{(\mu,2s)})\otimes_{\mathcal{O}_{n+1}} \mathbb{Q} = (\bar{g}_{\lambda, \mu}^{\nu}-g_{\lambda_{[n+1]}, \mu_{[n+1]}}^{\nu_{[n+1]}}) (\Delta_{2r,\mathbb{Q}}^{n+1}(\lambda) \otimes \Delta_{2s,\mathbb{Q}}^{n+1} (\mu)).$$
Now by Proposition \ref{finalprop} we have
\begin{eqnarray*}
\dim_\mathbb{Q} \Delta_{2p,\mathcal{O}_n}^{z,n}(\nu)\otimes_{\mathcal{O}_n} \mathbb{Q} &=& \dim_{\mathbb{F}} \Delta_{2p,\mathbb{F}}^{z,n}(\nu) \\
&\geq & \dim_{\mathbb{F}} \Delta_{2p,\mathbb{F}}^{z,n+1} (\nu) \\
&=& \dim_{\mathbb{Q}} \Delta_{2p,\mathcal{O}_{n+1}}^{z,n+1}(\nu) \otimes_{\mathcal{O}_{n+1}} \mathbb{Q}.
\end{eqnarray*}
It follows that 
$$\dim_\mathbb{Q} \Delta_{2p,\mathcal{O}_n}^{z,n}(\nu)(F_{(\lambda,2r)}\otimes F_{(\mu,2s)})\otimes_{\mathcal{O}_n} \mathbb{Q}  \geq \dim_{\mathbb{Q}} \Delta_{2p,\mathcal{O}_{n+1}}^{z,n+1}(\nu) (F_{(\lambda,2r)}\otimes F_{(\mu,2s)})\otimes_{\mathcal{O}_{n+1}} \mathbb{Q}.$$
Now, as 
$$\dim_{\mathbb{Q}} \Delta_{2r,\mathbb{Q}}^n(\lambda) \otimes \Delta_{2s,\mathbb{Q}}^n(\mu) = \dim_{\mathbb{Q}} \Delta_{2r,\mathbb{Q}}^{n+1}(\lambda) \otimes \Delta_{2s,\mathbb{Q}}^{n+1}(\mu),$$
we get
$$\bar{g}_{\lambda, \mu}^\nu - g_{\lambda_{[n]}, \mu_{[n]}}^{\nu_{[n]}} \geq \bar{g}_{\lambda, \mu}^\nu - g_{\lambda_{[n+1]}, \mu_{[n+1]}}^{\nu_{[n+1]}}$$
and hence  $g_{\lambda_{[n]}, \mu_{[n]}}^{\nu_{[n]}} \leq g_{\lambda_{[n+1]}, \mu_{[n+1]}}^{\nu_{[n+1]}}$ as required.
\end{proof}

%

 \bibliographystyle{amsalpha}

\def\cprime{$'$} \def\cprime{$'$}
\providecommand{\bysame}{\leavevmode\hbox to3em{\hrulefill}\thinspace}
\providecommand{\MR}{\relax\ifhmode\unskip\space\fi MR }
\providecommand{\MRhref}[2]{\href{http://www.ams.org/mathscinet-getitem?mr=#1}{#2}
}
\providecommand{\href}[2]{#2}

\end{document}